\RequirePackage{fix-cm}
\def\arxiv{1}
\ifdefined\arxiv
\documentclass[smallextended,a4paper,12pt]{svjour3-arxiv}
\else
\documentclass[smallextended]{svjour3}
\fi
\smartqed

\usepackage[sort&compress,numbers]{natbib}
\usepackage{tabls,booktabs,mdframed}
\usepackage{multirow,mathtools,wrapfig}
\usepackage{amsmath,amssymb}
\usepackage{amsfonts}
\usepackage{url}
\usepackage[pointedenum]{paralist}
\usepackage{cases,xspace,color}

\usepackage[breaklinks=true,pdfstartview=FitH,pagebackref=true]{hyperref}
\usepackage{subcaption}
\ifdefined\arxiv
\usepackage[ruled,noend]{algorithm2e}
\else
\usepackage[ruled,noend,linesnumbered]{algorithm2e}
\fi
\usepackage{graphicx,grffile}
\ifdefined\arxiv
\usepackage[margin=.8in]{geometry}
\fi
\usepackage[capitalize,nameinlink]{cleveref}
\crefformat{equation}{\textup{#2(#1)#3}}
\crefrangeformat{equation}{\textup{#3(#1)#4--#5(#2)#6}}
\crefmultiformat{equation}{\textup{#2(#1)#3}}{ and \textup{#2(#1)#3}}
{, \textup{#2(#1)#3}}{, and \textup{#2(#1)#3}}
\crefrangemultiformat{equation}{\textup{#3(#1)#4--#5(#2)#6}}%
{ and \textup{#3(#1)#4--#5(#2)#6}}{, \textup{#3(#1)#4--#5(#2)#6}}{, and \textup{#3(#1)#4--#5(#2)#6}}

\Crefformat{equation}{#2Equation~\textup{(#1)}#3}
\Crefrangeformat{equation}{Equations~\textup{#3(#1)#4--#5(#2)#6}}
\Crefmultiformat{equation}{Equations~\textup{#2(#1)#3}}{ and \textup{#2(#1)#3}}
{, \textup{#2(#1)#3}}{, and \textup{#2(#1)#3}}
\Crefrangemultiformat{equation}{Equations~\textup{#3(#1)#4--#5(#2)#6}}%
{ and \textup{#3(#1)#4--#5(#2)#6}}{, \textup{#3(#1)#4--#5(#2)#6}}{, and \textup{#3(#1)#4--#5(#2)#6}}

\crefdefaultlabelformat{#2\textup{#1}#3}

\newcommand{\dist}{\mbox{\rm dist}}
\newcommand{\prox}{\mbox{\rm prox}}
\newcommand{\relint}{\mbox{\rm relint}}

\newcommand{\norm}[1]{{\left\|{#1}\right\|}}
\newcommand{\R}{\mathbb{R}}

\newcommand{\N}{\mathbb{N}}
\renewcommand{\H}{\mathcal{E}}
\renewcommand{\SS}{\modify{\mathcal{S_{++}}}}
\renewcommand{\S}{\modify{\mathcal{S_{+}}}}
\newcommand{\M}{\mathcal{M}}

\newcommand{\levbet}{\text{Lev}\left({\xi}\right)}

\newcommand{\levv}[1]{\text{Lev}({#1})}
\newcommand{\dom}{\text{dom}}
\newcommand{\inprod}[2]{\langle#1,\,#2\rangle}
\DeclareMathOperator*{\argmin}{\arg\,min}

\newcommand*{\lmss}{\fontfamily{lmss}\selectfont}
\def\ipvm{{\sf ISQA}\xspace}
\def\ipn{{\sf ISQA$^{+}$}\xspace}
\def\a9a{{\lmss a9a}\xspace}
\def\w8a{{\lmss w8a}\xspace}
\def\ijcnn1{{\lmss ijcnn1}\xspace}
\def\cov{{\lmss covtype.scale}\xspace}
\def\realsim{{\lmss realsim}\xspace}
\def\news{{\lmss news20}\xspace}
\def\eps{{\lmss epsilon}\xspace}
\def\webspam{{\lmss webspam}\xspace}
\def\rcvt{{\lmss rcv1-test}\xspace}

\ifdefined\arxiv
\newcommand{\modify}[1]{{#1}}
\newcommand{\modifyy}[1]{{#1}}
\else
\newcommand{\modify}[1]{{{#1}}}
\newcommand{\modifyy}[1]{{\color{red}{#1}}}
\fi
\newcommand{\TheTitle}{Accelerating Inexact Successive Quadratic
Approximation for Regularized Optimization Through Manifold Identification}
\newcommand{\ShortTitle}{Manifold Identification of Inexact
Proximal-Newton-Type Methods}

\ifpdf
\hypersetup{
  pdftitle={\TheTitle},
  pdfauthor={Ching-pei Lee}
}
\fi
\title{{\TheTitle}\thanks{Version of \today.
	This work was supported in part by NSTC of R.O.C. grants
109-2222-E-001-003-MY3 and 111-2628-E-001-003-, and Academia Sinica
grant AS-GCS-111-M05.}}
\titlerunning{{\ShortTitle}}

\author{Ching-pei Lee}
\institute{Ching-pei Lee \at
	\email{\url{leechingpei@gmail.com}}.
Institute of Statistical Mathematics, Tokyo, Japan}

\date{}
\journalname{Mathematical Programming}
\begin{document}

\maketitle

\begin{abstract}
For regularized optimization that minimizes the sum of a smooth
term and a regularizer that promotes structured solutions,
inexact proximal-Newton-type methods, or successive quadratic approximation
(SQA) methods, are widely used for their superlinear convergence in
terms of iterations.
However, unlike the counter parts in smooth optimization, they suffer
from lengthy running time in solving regularized subproblems because
even approximate solutions cannot be computed easily,
so their empirical time cost is not as impressive.
In this work, we first show that for partly smooth regularizers,
although general inexact solutions cannot identify the active manifold
that makes the objective function smooth,
approximate solutions generated by commonly-used subproblem solvers
will identify this manifold, even with arbitrarily low solution
precision.
We then utilize this property to propose an improved SQA method, \ipn,
that switches to efficient smooth optimization methods after this
manifold is identified.
We show that for a wide class of degenerate solutions,
\ipn possesses superlinear convergence not only in iterations,
but also in running time because the cost per iteration is bounded.
In particular, our superlinear convergence result holds on problems
satisfying a sharpness condition that is more general than that in existing
literature.
\modifyy{We also prove iterate convergence under a sharpness condition for
inexact SQA, which is novel for this family of methods that could
easily violate the classical relative-error condition frequently used
in proving convergence under similar conditions.}
\ifdefined\arxiv
Experiments on real-world problems confirm that \ipn greatly
improves the state of the art for regularized optimization.
\else
\modify{Experiments on real-world problems support that \ipn
improves running time over some modern solvers for regularized
optimization.}
\fi

\keywords{Variable metric \and Manifold identification \and
Regularized optimization \and Inexact method \and Superlinear
convergence}
\end{abstract}
\subclass{49M15 \and 90C55 \and 49K40 \and 90C31 \and 90C55 \and 65K05}

\section{Introduction} \label{sec:intro}
Consider the following regularized optimization problem:
\begin{equation}
	\min_{x \in \H}\quad F(x) \coloneqq f(x) + \Psi(x),
	\label{eq:f}
\end{equation}
where $\H$ is a Euclidean space with an inner product
$\inprod{\cdot}{\cdot}$ and its induced norm $\norm{\cdot}$,
the regularizer $\Psi$ is extended-valued, convex, proper, and
lower-semicontinuous, $f$ is \modify{continuously differentiable with
Lipschitz-continuous gradients},
and the solution set $\Omega$ is non-empty.
This type of problems is ubiquitous in applications such as machine
learning and signal processing \modify{(see, for example,
\cite{Tib96a,CanR09a,CanRT06a})}.
One widely-used method for \cref{eq:f} is inexact successive quadratic
approximation (\ipvm).
At the $t$th iteration with iterate $x^t$,
\ipvm obtains the update direction $p^t$ by approximately solving
\begin{gather}
\label{eq:quadratic}
p^t \approx \argmin_{p \in \H}\quad Q_{H_t}\left( p; x^t
	\right),\\
\label{eq:Qdef}
Q_{H_t}\left( p; x^t \right) \coloneqq \inprod{\nabla f\left( x^t
	\right)}{p} + \frac12 \inprod{p}{ H_t p} + \Psi\left( x^t + p
	\right) - \Psi\left( x^t \right),
\end{gather}
where $H_t$ is a self-adjoint positive-semidefinite linear
endomorphism of $\H$.
The iterate is then updated along $p^t$ with a step size $\alpha_t
> 0$.

\ipvm is among the most efficient for \cref{eq:f}.
Its variants differ in the choice of $H_t$ and $\alpha_t$, and how
accurately \cref{eq:quadratic} is solved.
In this class, proximal Newton (PN) \citep{LeeSS14a,LiAV17a} and proximal
quasi-Newton (PQN) \citep{SchT16a} are popular for their fast convergence in
iterations.
Regrettably, their subproblem has no closed-form solution as $H_t$ is
non-diagonal,
so one needs to use an iterative solver for \cref{eq:quadratic} and
the running time to reach the accuracy requirement can hence be
lengthy.
For example, to attain the same superlinear convergence as truncated
Newton for smooth optimization, PN that takes $H_t = \nabla^2 f$
requires increasing accuracies for the subproblem solution, implying a
growing and unbounded number of inner iterations of the subproblem
solver.
Its superlinear convergence thus gives little practical advantage in
running time.
On the contrary, in smooth optimization, one can solve
\cref{eq:quadratic} with a bounded cost by either conjugate gradient
(CG) or matrix factorizations since $\Psi \equiv 0$.
The advantage of second-order methods over the first-order ones in
regularized optimization is therefore not as significant as that in
smooth optimization.

A possible remedy when $\Psi$ is partly smooth \citep{Lew02a}
is to switch to smooth optimization after identifying an active
manifold $\M$ that contains a solution \modify{$\hat x$} to \cref{eq:f}
and makes $\Psi$ confined to it smooth.
We say an algorithm can identify $\M$ if there is a neighborhood $U$
of \modify{$\hat x$} such that $x^t \in U$ implies $x^{t+1} \in \M$, and call it
possesses the manifold identification property.
Unfortunately, for \ipvm, this property in general only holds when
\cref{eq:quadratic} is always solved exactly.
Indeed, even if each $p^t$ is arbitrarily close to the corresponding
exact solution, it is possible that no iterate lies in the active
manifold, as shown below.
\begin{example}
\label{ex:conv}
Consider the following simple example of \cref{eq:f} with
$\Psi(\cdot) = \|\cdot\|_1$:
\begin{equation*}
	\min_{x \in \R^2}\, (x_1 - 2.5)^2 + (x_2 - 0.3)^2 + \|x\|_1,
\end{equation*}
whose only solution is $\modify{\hat x} =
(2,0)$, and $\|x\|_1$ is smooth relative to $\M = \{x \mid x_2 =
0\}$ around $\modify{\hat x}$.
Consider $\{x^t\}$ with $x^t_1 = 2 +
f(t), x^t_2 = f(t)$, for some $f(t) > 0$ with $ f(t) \downarrow 0$,
and let $H_t \equiv I, \alpha_t \equiv 1$, and $p^t = x^{t+1} - x^t$.
The optimum of \cref{eq:quadratic} is $p^{t*} = \modify{\hat x} - x^t$, so $\|x^t-
\modify{\hat x}\| = O(f(t))$  and $\|p^t - p^{t*}\|=O(f(t))$.
As $f$ is arbitrary, both the subproblem approximate
solutions and their corresponding objectives converge to the
optimum arbitrarily fast, but $x^t \notin \M$ for all $t$.
\qed
\end{example}
Moreover, some versions of inexact PN generalize the
stopping condition of CG for truncated Newton to require $\|r^t\|
\rightarrow 0$, where
\begin{equation}
r^t \coloneqq \argmin_{r} \|r\|, \quad \text{ subject to } \quad r \in
\partial_{p} Q_{H_t}(p^t;x^t),
\label{eq:subnorm}
\end{equation}
but \cref{ex:conv} gives a sequence of $\{\|r^t\|\}$ converging to
$1$,  hinting that such a condition might have an implicit
relation with manifold identification.

Interestingly, in our numerical experience in
\cite{LeeW18a,LeeLW19a,YSL20a}, \ipvm with approximate subproblem
solutions, even without an increasing solution precision and on
problems that are not strongly convex, often identifies the active
manifold rapidly.
We thus aim to provide theoretical support for such a phenomenon
and utilize it to devise more efficient and practical methods that
trace the superior performance of second-order methods in smooth
optimization.

In this work, we show that \ipvm essentially possesses the manifold
identification property, by giving a sufficient condition for inexact
solutions of \cref{eq:quadratic} in \ipvm to identify the active
manifold that is satisfied by the output of most of the widely-used
subproblem solvers even if \cref{eq:quadratic} is solved arbitrarily
roughly.
We also show that $\norm{r^t} \downarrow 0$ is indeed sufficient for
manifold identification, so PN can achieve superlinear convergence in
a more efficient way through this property.
When the iterates do not lie in a compact set, it is possible that the
iterates do not converge, in which case even algorithms possessing the
manifold identification property might fail to identify the active
manifold because the iterates never enter a neighborhood that enables
the identification.
\modifyy{Therefore, we also show convergence of the iterates under a sharpness
condition \modify{widely-seen} in real-world problems that generalizes
the quadratic growth condition and the weak sharp minima.
Under convexity, this sharpness condition is equivalent to a type of
the Kurdyka-{\L}ojasiewicz (KL) condition \citep{Kur98a,Loj63a}, but
convergence of general \ipvm methods under the KL condition is
unknown since the inexactness condition can easily violate the
relative-error condition needed in \cite{AttBRS10a,BonLPPR17a}, and
thus our analysis provides a novel approach to obtain iterate
convergence for this family of algorithms.}
Based on these results, we propose an improved, practical
algorithm \ipn that switches to smooth optimization after the active
manifold is presumably identified.
We show that \ipn is superior to existing PN-type methods as it
possesses the same superlinear and even quadratic rates in
iterations but has bounded per-iteration cost.
\ipn hence also converges superlinearly in running time,
which, to our best knowledge, is the first of the kind.
Our analysis is more general than existing ones in guaranteeing
superlinear convergence in a broader class of degenerate problems.
Numerical results also confirm \ipn's much improved efficiency
over PN and PQN.

\subsection{Related Work}
\label{sec:related}
\ipvm for \cref{eq:f} or the special case of constrained optimization
has been well-studied, and we refer the readers to
\cite{LeeW18a} for a detailed review of related methods.
We mention here in particular the works
\cite{LeeSS14a,ByrNO16a,LiAV17a,MCY19a} that provided superlinear convergence
results.
\citet{LeeSS14a} first analyzed the superlinear convergence of
PN and PQN.
Their analysis considers only strongly convex $f$,
so both the convergence of the iterate and the positive-definiteness
of the Hessian are guaranteed.
Their inexact version requires $\|r^t\|\downarrow 0$, which
might not happen when the solutions to \cref{eq:quadratic} are only
approximate, as illustrated in \cref{ex:conv}.
With the same requirement for $\|r^t\|$ as \cite{LeeSS14a},
\citet{LiAV17a} showed that superlinear convergence for inexact PN can
be achieved when $f$ is self-concordant.
\citet{ByrNO16a} focused on $\Psi(\cdot) = \|\cdot\|_1$ and showed
superlinear convergence of PN under the subproblem
stopping condition \cref{eq:gm} defined in \cref{sec:background},
which is achievable as long as $p^t$ is close enough to the optimum of
\cref{eq:quadratic}.
To cope with degenerate cases in which the Hessian is only positive
semidefinite,
\citet{MCY19a} used the stopping condition of \cite{ByrNO16a} to
propose a damping PN for general $\Psi$ and showed that its
iterates converge and achieve superlinear convergence under convexity
and the error-bound (EB) condition \citep{LuoT92a} even if $F$ is not
coercive.
A common drawback of \cite{LeeSS14a,ByrNO16a,LiAV17a,MCY19a} is that they
all require increasing precisions in \modifyy{solving the subproblem}, so the superlinear
rate is observed only in iterations but not in running time in their
experiments.
In contrast, by switching to smooth optimization after identifying the
active manifold, \ipn achieves superlinear convergence not only in
iterations but also in time, and is thus much more efficient in
practice.
Our superlinear convergence result also allows a broader range of
degeneracy than that in \cite{MCY19a}.

Although \ipvm is intensively studied, its ability for manifold
identification is barely discussed because this does not in general
hold, as noted in \cref{ex:conv}.
\citet{Har11a} showed that \ipvm identifies the active manifold under
the \ifdefined\arxiv
impractical
\fi
assumptions that \cref{eq:quadratic} is always solved exactly and
the iterates converge,
and his analysis cannot be extended to inexact versions.
Our observation in \cite{LeeW18a,LeeLW19a,YSL20a} that \ipvm
identifies the active manifold empirically motivated this work to
provide theoretical guarantees for this phenomenon.

Manifold identification requires the iterates, or at least a
subsequence, to converge to a point of partial smoothness.
In most existing analyses for \cref{eq:f}, iterate convergence is
proven under either:
(i) $f$ is convex and the algorithm is a first-order one, (ii) $F$ is
strongly convex,  or (iii) the Kurdyka-{\L}ojasiewicz (KL) condition
holds.
Analyses for the first scenario rely on the implicit regularization of
first-order methods such that their iterates lie in a bounded region
\citep{LiaFP17a}, but this is not applicable to \ipvm.
Under the second condition, convergence of the objective directly
implies that of the iterates.
\modifyy{For the third case, convergence of the full iterate sequence is
usually proven under an assumption of a relative-error behavior, of
the form
\[
	\min_{v \in \partial F(x^{t+1})} \norm{v} \leq b \norm{x^{t+1}
- x^t}, \forall t
\]
for some $b > 0$,
as done in \cite{ChoPR14a,BonPR20a,AttBRS10a}, but this condition can
easily be violated when inexactness kicks in in \ipvm, as argued by
\citet{BonPR20a}.
To work around this issue, \cite{BonPR20a} further assumed that
the forward-backward envelope \citep{SteTP17a} of $F$ satisfies the KL
condition and obtained iterate convergence under such a situation, but
whether KL condition of $F$ implies that of its forward-backward
envelope is unclear.
The only exception to get convergence under the KL condition of $F$
for a specific type of SQA method is} \cite{MCY19a} that shows the
convergence of the iterates for their specific algorithm under EB and
convexity of $f$ but \cite{MCY19a} requires $H_t$ in
\cref{eq:quadratic} to be the Hessian of $f$ plus a multiple of
identity and their analysis cannot be extended to general $H_t$.
On the other hand, our analysis for iterate convergence is novel and
more general in covering a much broader algorithmic framework and
requiring only a general sharpness condition for $F$ that contains
both EB and the weak-sharp minima \citep{BurF93a} as special cases.

Our design of the two-stage \ipn is inspired by
\cite{LeeW12a,YSL20a} in conjecturing that the active manifold is
identified after the current manifold remains unchanged,
but the design of the first stage is quite different and we
also add in additional safeguards in the second stage.
\citet{LeeW12a} used dual averaging in the first
stage for optimizing the expected value of an objective function
involving random variables, so their algorithm is more suitable for
stochastic settings.
\citet{YSL20a} focused on distributed optimization and their
usage of manifold identification is for reducing the communication
cost, instead of accelerating general regularized optimization
considered in this work.

\subsection{Outline}
This work is outlined as follows.
In \cref{sec:background}, we describe the algorithmic framework
and give preliminary properties.
Technical results in \cref{sec:manifold} prove the manifold
identification property of \ipvm and the convergence of the iterates.
We then describe the proposed \ipn and show its superlinear
convergence in running time in \cref{sec:alg}.
The effectiveness of \ipn is then illustrated through extensive
numerical experiments in \cref{sec:exp}.
\cref{sec:con} finally concludes this work.
Our implementation of the described algorithms is available at
\url{https://www.github.com/leepei/ISQA_plus/}.

\section{Preliminaries}
\label{sec:background}
We denote the minimum of \cref{eq:f} by $F^*$; the domain of $\Psi$
by $\dom(\Psi)$; and the set of convex, proper, and lower
semicontinuous functions by $\Gamma_0$.
\modify{For any set $C$, $\relint(C)$ denotes its relative interior.}
We will frequently use the following notations.
\[
	\delta_t \coloneqq F(x^t) - F^*,\quad
P_{\Omega}(x) \coloneqq \argmin_{y \in \Omega} \|x - y\|,\quad
	\dist(x,\Omega) \coloneqq \|x - P_\Omega(x)\|.
\]
The level set $\levbet \coloneqq \left\{ x\mid  F(x) - F^* \leq \xi
\right\}$ for any $\xi \geq 0$  is closed but not necessarily
\modify{bounded.}
A function is $L$-smooth if it is \modify{differentiable with the gradient
$L$-Lipschitz continuous}.
We denote the identity operator by $I$.
For self-adjoint linear endomorphisms $A,B$ of $\H$,
$A \succ B$ ($\succeq$) means $A-B$ is positive definite (positive
semidefinite).
We abbreviate $A \succ \tau I$ to $A \succ \tau$ for $\tau \in \R$.
The set of $A$ with $A \succ 0$ is denoted by $\SS$.
The subdifferential $\partial \Psi(x)$ of $\Psi$ at $x$ is
well-defined as $\Psi \in \Gamma_0$,
hence so is the generalized gradient $\partial F(x) = \nabla f(x) +
\partial \Psi(x)$.
For any $g \in \Gamma_0$, $\tau\geq 0$, and \modify{$\Lambda \in \SS$},
the proximal mapping
\begin{equation}
\label{eq:proximal}
	\prox^{\Lambda}_{\tau g}(x) \coloneqq \argmin_{y \in \H}\, \frac12
	\inprod{x - y}{\Lambda (x-y)} + \tau g\left( y \right)
\end{equation}
is continuous and finite in $\H$ even outside
$\dom(g)$.
When $\Lambda = I$, \cref{eq:proximal} is shorten to
$\prox_{\tau g}(x)$.
For \cref{eq:quadratic}, we denote its optimal solution by
$p^{t*}$. When there is no ambiguity, we abbreviate $Q_{H_t}(\cdot;
x^t)$ to $Q_t(\cdot)$, $Q_t(p^t)$ to $\hat Q_t$, and $Q_t(p^{t*})$ to
$Q_t^*$.
\ifdefined\arxiv
\else
\modify{Some notations used in this paper are summarized in
	\cref{tbl:notation}.}

\begin{table}[bht]
	\modify{
		\caption{\modify{Summary of Some Notations.}}
	\label{tbl:notation}
	\centering
	\begin{tabular}{@{}ll@{}}
		\toprule
		Notation & Meaning \\
		\midrule
		$F$ & Objective function in \cref{eq:f}\\
		$F^*$ & Optimal objective value\\
		$f$ & The smooth term in \cref{eq:f}\\
		$\Psi$ & The partly smooth term in \cref{eq:f}\\
		$\Omega$ & Solution set of \cref{eq:f}\\
		$P_\Omega$ & Projection onto $\Omega$\\
		$\SS (\S)$ & The set of positive (semi)definite self-adjoint
		linear endomorphisms of $\H$\\
		$\Gamma_0$ & The set of convex, proper, and lower-semicontinuous functions\\
		$\prox^\Lambda_{\tau g}$ & \modifyy{The proximal mapping
		defined in} \cref{eq:proximal}\\
		$x^t$ & Iterate at the $t$th iteration \\

		$Q_{H_t}(p;x)$ & The subproblem of \ipvm defined in \cref{eq:Qdef}\\
		\modifyy{$H_t$} & \modifyy{The quadratic term in defining the subproblem of \ipvm
			in \cref{eq:Qdef}}\\
		$Q_t(\cdot)$ & $Q_{H_t}(\cdot;x^t)$\\
		$p^{t*}$ & $\argmin_p\, Q_t(p)$\\
		$Q_t^*$ & $Q_t(p^{t*})$\\
		$p^t$ & Update direction at the $t$th iteration generated by approximately solving
		\cref{eq:quadratic}\\
		$\hat Q_t$ & $Q_t(p^t)$\\
		$\alpha_t$ & Step size at the $t$th iteration\\
		$r^t$ & \modifyy{The minimum-norm subgradient of
			$Q_{H_t}(\cdot; x^t)$;
		see} \cref{eq:subnorm}\\
		$\M$ & Active manifold; see \cref{def:ps}\\
		$\relint(C)$ & Relative interior of $C$\\
		$\phi_t$ & A local description of $\M$ around $x^t$\\
		$y^t$ & A point satisfying $\phi_t(y^t) = x^t$\\
		$g^t$ & $\nabla F(\phi_t(y^t))$\\
		$\mu_t$ & \modifyy{The damping term in the smooth step; see} \cref{eq:newton}\\
		$F_\phi(\cdot)$ & $F(\phi(\cdot))$ \\
		$y^*$ & A point satisfying $F_{\phi}(y^*) = x^*$\\
		$d_t$ & $\norm{y^t - y^*}$\\
		\bottomrule
\end{tabular}
}
\end{table}
\fi

\subsection{Algorithmic Framework}
\label{subsec:alg}
We give out details in defining the \ipvm framework
by discussing the choice of $H_t$,
the subproblem solver and its stopping condition, and how sufficient
objective decrease is ensured.
We consider the algorithm a two-level loop procedure, where the outer
loop updates the iterate $x^t$ and the iterations of the subproblem
solver form the inner loop.

After obtaining $p^t$ from \cref{eq:quadratic}, we need to find a
step size $\alpha_t > 0$ for it  to ensure sufficient objective
decrease.
Given $\gamma,\beta \in (0,1)$, we take $\alpha_t$ as the
largest value in $\{\beta^0,\beta^1,\dotsc\}$ satisfying an
Armijo-like condition.
\begin{equation}
	F\left( x^t + \alpha_t p^t \right) \leq F\left( x^t \right) +
	\gamma \alpha_t Q_t(p^t).
	\label{eq:armijo}
\end{equation}
This condition is satisfied by all $\alpha_t$ small enough as long as
$Q_t(p^t) < 0$ and $Q_t(\cdot)$ is strongly convex; see
\cite[Lemma~3]{LeeW18a}.

For the choice of $H_t$,
we only make the following blanket assumption without further
specification to make our analysis more general.
\begin{equation}
\exists M,m > 0,\quad \text{ such that } \quad M \succeq H_t  \succeq
m, \;\forall t \geq 0.
\label{assum:Hbound}
\end{equation}

For \cref{eq:quadratic}, any suitable solver for regularized
optimization, such as (accelerated) proximal gradient,
(variance-reduced) stochastic gradient methods, and their variants,
can be used, and the following are common for their inner loop
termination:
\begin{gather}
	\hat Q_t - Q_t^* \leq \epsilon_t,
	\label{eq:subprob}\\
\left\| r^t \right\| \leq \epsilon_t,
\quad \text{ or }
	\label{eq:partial}\\
G^\tau_{H_t} \left( p^t;x^t \right)\coloneqq \norm{p^t - \bar
p^t_\tau}\leq \epsilon_t,
	\label{eq:gm}
\end{gather}
for some given $\epsilon_t\geq 0$ and $\tau > 0$, where
\begin{align}
\label{eq:p}
\bar p^t_\tau &\coloneqq \prox_{\tau \Psi}\left(\left(x^t + p^t \right) -
	\tau \left(\nabla f(x^t) + H_t p^t\right)\right) - x^t.
\end{align}
\modify{
The point $\bar p^t_\tau$ in \cref{eq:p} is computed by taking a
proximal gradient step of the subproblem \cref{eq:quadratic} from
$p^t$, and thus $\bar p^t - p^t$ is \modifyy{the proximal gradient}
(with step size $\tau$) of $Q_t$ at $p^t$.
Because the subproblem is strongly convex, the norm of the proximal
gradient is zero at $p^t$ if and only if $p^t$ is the unique solution
to the subproblem.
We will see in the next subsection that the norm of this proximal
gradient squared is also equivalent to the objective distance to the
optimum of the subproblem.}
We summarize this framework in \cref{alg:vm}.

\begin{algorithm}[h]
\DontPrintSemicolon
\SetKwInOut{Input}{input}\SetKwInOut{Output}{output}
\caption{A Framework of \ipvm for \cref{eq:f}}
\label{alg:vm}
\Input{$x^0 \in \H$, $\gamma, \beta \in (0,1)$}

\For{$t=0,1,\dotsc$}{
	$\alpha_t \leftarrow 1$,
	pick $\epsilon_t \geq 0$ and $H_t$, and solve \cref{eq:quadratic}
	for $p^t$ satisfying \cref{eq:subprob},
	\cref{eq:partial}, or \cref{eq:gm}

	\lWhile{\cref{eq:armijo} is not satisfied}
	{$\alpha_t \leftarrow \beta \alpha_t$}
	$x^{t+1} \leftarrow x^t + \alpha_t p^t$
}
\end{algorithm}

\subsection{Basic Properties}
Under \cref{assum:Hbound}, \cref{eq:Qdef} is $m$-strongly convex with
respect to $p$, so the following standard results hold for any $\tau
\in (0, 1/M]$ and any $p^t \in \H$ \citep{Nes18a,DruL16a}.
\begin{gather}
\|r^t\|^2 \geq 2 m \left( Q_t(p^t) - Q_t^* \right)
\geq \frac{m}{M} G^\tau_{H_t}\left(
	p^t;x^t \right)^2,
\label{eq:PG}\\
G^\tau_{H_t}\left( p^t;x^t \right)^2 \geq \tau \left( \frac{(2m^{-1} +
	\tau)(1 + M \tau)}{\tau} - \frac{1}{2} \right)^{-1}
	\left(Q_t\left( \bar p^t_\tau\right) - Q_t^* \right).
\label{eq:EB}
\end{gather}
Therefore, \cref{eq:gm} and \cref{eq:subprob} are almost equivalent
and implied by \cref{eq:partial}, while \cref{ex:conv} has shown that
\cref{eq:partial} is a stronger condition not implied by
\cref{eq:subprob}.
Although \cref{eq:EB} does not show that \cref{eq:gm} directly
implies \cref{eq:subprob}, once \cref{eq:gm} is satisfied, we can use
it to find $\bar p^t_\tau$ satisfying \cref{eq:subprob} from $p^t$.

A central focus of this work is manifold identification, so we first
formally define manifolds following \cite{Vai84a}.
A set $\M \in \R^m$ is a $p$-dimensional $\mathcal{C}^k$ manifold
around $x \in \R^m$ if there is a $\mathcal{C}^k$ function $\Phi: \R^m
\rightarrow \R^{m-p}$ whose derivative at $x$ is surjective such that for
$y$ close enough to $x$, $y \in \M$ if and only if $\Phi(y) = 0$.
Through the implicit function theorem, we can also use a
$\mathcal{C}^k$ parameterization $\phi: \R^p \rightarrow \M$, with
$\phi(y) = x$ and the derivative injective at $y$, to describe a
neighborhood of $x$ on $\M$.
In our definition and proof below, we will require the function being
(subdifferentially) regular and prox-regular at a point $x$.
These two conditions are satisfied by convex functions everywhere and
thus we do not provide their definitions here, and interested readers
can refer to \cite[][Definition~7.25, Definition~13.27]{RocW09a} for details.
Now we are ready for the definition of partial smoothness
\citep{Lew02a} that we assume for the regularizer when discussing
manifold identification.
\begin{definition}[Partly smooth]
	\label{def:ps}
	A function $\Psi$ is partly
	smooth at a point $x^*$ relative to a set $\M$ containing $x^*$ if
	$\partial \Psi(x^*) \neq \emptyset$ and:
	\begin{enumerate}
			\item Around $x^*$, $\M$ is a $\mathcal{C}^2$-manifold and
				$\Psi|_{\M}$ is $\mathcal{C}^2$.
		\item The affine span of $\partial \Psi(x)$ is a translate
			of the normal space to $\M$ at $x^*$.
		\item $\partial \Psi$ is continuous at $x^*$ relative to $\M$.
		\item \modify{$\Psi$ is regular at all points $x \in \M$ with
			$\partial \Psi(x) \neq  \emptyset$
		around $x^*$.}
	\end{enumerate}
\end{definition}
Intuitively, it means $\Psi$ is smooth around
$x^*$ in $\M$ but changes drastically along directions leaving the
manifold.
We also call this $\M$ the active manifold.

\modify{As the original identification results in \cite[Theorem~5.3]{HarL04a}
and \cite[Theorem~4.10]{LewZ13a} require the sum $F$ to be partly smooth
but our setting does not require so for $f$,
we first provide a result relaxing the conditions to ensure
identification in our scenario.}
\modify{
\begin{lemma}
Consider \cref{eq:f} with $f\in \mathcal{C}^1$ and $\Psi$
partly
smooth at a point $x^*$ relative to a $\mathcal{C}^2$-manifold
$\M$.
If $\Psi$ is prox-regular at $x^*$,
the nondegenerate condition
\begin{equation}
0 \in \relint \left(\partial F(x^*)\right) = \nabla f(x^*) + \relint
\left(\partial \Psi(x^*)\right)
\label{eq:nod}
\end{equation}
holds, and there is a sequence $\{x^t\}$ converging to $x^*$ with
$F(x^t) \rightarrow F(x^*)$,
then
\[
	\dist\left( 0, \partial F(x^t) \right) \rightarrow 0
	\quad \Leftrightarrow \quad x^t \in \M \text{ for all t large}.
\]
\label{lemma:identify0}
\end{lemma}
}
\modify{
\begin{proof}
We first observe that as $f \in \mathcal{C}^1$, $x^t \rightarrow x^*$
implies $f(x^t) \rightarrow f(x^*)$,
whose combination with $F(x^t) \rightarrow F(x^*)$ further implies
$\Psi(x^t) \rightarrow \Psi(x^*)$.
Therefore, the premises of \cite[Theorem~4.10]{LewZ13a} on $\Psi$ are
satisfied.
We then note that
\begin{equation}
	\label{eq:convdist}
	\dist\left(0, \partial F(x^t)\right) \rightarrow 0 \quad
	\Leftrightarrow \quad
	\dist\left(-\nabla f(x^t), \partial \Psi(x^t)\right) \rightarrow 0.
\end{equation}
Again from that $f \in \mathcal{C}^1$, $x^k \rightarrow x^*$ implies
$\nabla f(x^k) \rightarrow \nabla f(x^*)$,
so by \cref{eq:nod}, \cref{eq:convdist} is
further equivalent to
\[
	\dist\left(-\nabla f(x^*), \partial \Psi(x^t)\right) \rightarrow
	0,
\]
which is the necessary and sufficient condition for $x^t \in \M$ for
all $t$ large in \cite[Theorem~4.10]{LewZ13a} because \cref{eq:nod}
indicates that $- \nabla f(x^*) \in \relint(\partial \Psi(x^*))$.
We then apply that theorem to obtain the desired result.
Here we note that for the requirements
\qed
\end{proof}
}

\modify{
Using \cref{lemma:identify0}, we further state an identification
result for \cref{eq:f} under our setting without the need to check
whether $\{F(x^t)\}$ converges to $F(x^*)$.
This will be useful in our later theoretical development.}
\begin{lemma}
	Consider \cref{eq:f} with \modify{$f \in \mathcal{C}^1$} and $\Psi \in
\Gamma_0$.
If $\Psi$ is partly smooth relative to a manifold $\M$ at a point
\modify{$x^*$ satisfying \cref{eq:nod}}, and there is a sequence $\{x^t\}$
converging to $x^*$, then we have
\[
	\dist\left( 0, \partial F(x^t) \right) \rightarrow 0
	\quad \Rightarrow \quad x^t \in \M \text{ for all t large}.
\]
\label{lemma:identify}
\end{lemma}
\begin{proof}
As $\Psi \in \Gamma_0$, it is subdifferentially continuous at
$x^* \in \dom(\Psi)$ by \cite[Example 13.30]{RocW09a}.
Thus, $\dist(0, \partial F(x^t)) \rightarrow 0 \in \partial F(x^*)$
and $x^t \rightarrow x^*$ imply that $F(x^t) \rightarrow F(x^*)$.
The desired result is then obtained by applying
\modify{\cref{lemma:identify0}}.
\qed
\end{proof}
\modify{We note that the requirement of the above two lemmas is partial
	smoothness of $\Psi$, instead of $F$, at $x^*$.
Therefore, it is possible that $F\mid_{\M}$ is not $\mathcal{C}^2$, as
we only require $\Psi\mid_{\M} \in \mathcal{C}^2$ and $f$ being
$L$-smooth.}

\section{Manifold Identification of \ipvm}
\label{sec:manifold}
Our first major result is the manifold identification property of
\cref{alg:vm}.
We start with showing that the strong condition \cref{eq:partial} with
$\epsilon_t \downarrow 0$ is sufficient.
\begin{theorem}
Consider a point $x^*$ satisfying \cref{eq:nod} with $\Psi \in \Gamma_0$
partly smooth at $x^*$ relative to some manifold $\M$.
Assume $f$ is locally $L$-smooth
for $L > 0$ around $x^*$.
If \cref{alg:vm} is run with the condition \cref{eq:partial} and
\cref{assum:Hbound} holds,
then
there exist $\epsilon, \delta > 0$ such that
$\|\modify{x^t} - x^* \| \leq \delta, \epsilon_t \leq \epsilon$, and $\alpha_t = 1$
imply $x^{t+1} \in \M$.
\label{thm:identify1}
\end{theorem}
\begin{proof}
Since each iteration of \cref{alg:vm} is independent of the previous
ones, we abuse the notation to let $x^t$ be the input of
\cref{alg:vm} at the $t$th iteration and $p^t$ the corresponding
inexact solution to \cref{eq:quadratic}, but $x^{t+1}$ is
irrelevant to $p^t$ or $\alpha_t$.
Assume for contradiction the statement is false.
Then there exist a sequence $\{x^t\}\subset \H$ converging to $x^*$, a
nonnegative sequence $\{\epsilon_t\}$ converging to $0$, a sequence
$\{H_t\} \subset \SS$ satisfying \cref{assum:Hbound}, and a sequence
$\{p^t\}\subset \H$ such that $Q_{H_t}(p^t;x^t)$ in \cref{eq:Qdef}
satisfies
\cref{eq:partial} for all $t$, yet $x^t + p^t \notin \M$ for all $t$.
From \cref{eq:subnorm},
\begin{equation}
	r^t - \nabla f\left( x^t \right) - H_t p^t
	\in \partial \Psi\left( x^{t} + p^t \right), \quad \forall t \ge 0.
\label{eq:r}
\end{equation}
Therefore, we get from \cref{eq:r} that
\begin{align}
	\nonumber
	\dist\left(0, \partial F\left(x^t + p^t \right) \right) &=
	\dist\left(-\nabla f\left( x^t + p^t \right), \partial \Psi\left(
	x^t + p^t \right)\right)\\
	\nonumber
	&\leq \norm{-\nabla f\left( x^t + p^t\right) - \left(r^{t} - \nabla
	f\left( x^{t} \right) - H_{t} p^t \right)}\\
	\nonumber
	&\leq \norm{\nabla f\left( x^t + p^t\right) - \nabla f\left( x^{t}
	\right)} + \norm{r^{t}} +
	\norm{H_{t}} \norm{p^t}\\
	&\leq \left(L + M \right) \norm{p^t} + \norm{r^{t}}.
	\label{eq:dist}
\end{align}
From \cref{assum:Hbound} and the convexity of $\Psi$,
$Q_{H_t}(\cdot;x^t)$ is $m$-strongly convex, which implies
\begin{equation}
	Q_{H_t}\left(p;x^t \right) - Q_{H_t}\left( p^{t*};x^t \right) \geq
\frac{m}{2} \norm{p - p^{t*}}^2, \quad  \forall p \in \H.
\label{eq:qg}
\end{equation}
Combining \cref{eq:qg} and \cref{eq:PG} shows that
\begin{equation}
	\label{eq:pbound}
	m^{-1} \|r^t\| + \|p^{t*}\| \geq \|p^t\|.
\end{equation}
Since $\{x^t\}$ converges to $x^*$, by the argument in
\cite[Lemma~3.2]{Wri12a}, we get
\begin{equation}
p^{t*} = O(\|x^t - x^*\|)
\label{eq:pt}
\end{equation}
whenever $x^t$ is close enough to $x^*$.
Thus \cref{eq:pbound}, \cref{eq:pt}, \cref{eq:partial}, and that
$\epsilon_t \downarrow 0$ imply
\begin{equation}
\lim_{t \rightarrow \infty} \|p^t\| = 0,
\label{eq:pt2}
\end{equation}
Substituting \cref{eq:partial} and \cref{eq:pt2} into \cref{eq:dist}
gives $\dist(0, \partial F(x^t)) \rightarrow 0$, and by \cref{eq:pt2}
we also get $ x^t + p^t \rightarrow x^* + 0 = x^*$.  Therefore
\cref{lemma:identify} implies that
$x^{t} + p^{t} \in \M$ for all $t$ large enough, proving the
desired contradiction.
\qed
\end{proof}

\cref{thm:identify1} shows that if a variant of PN or PQN needs
$\|r^t\|\downarrow 0$ to achieve superlinear convergence, $\M$ will be
identified in the middle, so one can reduce the running time by
switching to smooth optimization that can be conducted more
efficiently while possessing the same superlinear convergence.
Moreover, although \cref{thm:identify1} shows that \cref{eq:partial}
is sufficient for identifying the active manifold,
it might never be satisfied as \modify{\cref{ex:conv}} showed.
We therefore provide another sufficient condition for \ipvm to
identify the active manifold that is satisfied by most of the
widely-used solvers for \cref{eq:quadratic},
showing that \ipvm essentially possesses the manifold identification
property.
This result uses the condition \cref{eq:subprob}, which is weaker than
\cref{eq:subnorm}, and we follow \cite{LeeW18a} to define
$\{\epsilon_t\}$ using a given $\eta \in [0,1)$:
\begin{equation}
	\epsilon_t = \eta \left( Q_t(0) - Q_t^* \right) = -\eta Q_t^*.
	\label{eq:multiplicative}
\end{equation}
As argued in \cite{LeeW18a} and practically adopted in various
implementations including \cite{LeeC17a,DunLGBHJ18a,LeeLW19a},
it is easy to ensure that \cref{eq:subprob} with
\cref{eq:multiplicative} holds for some $\eta <1$ under
	\cref{assum:Hbound} (although the explicit value might be unknown)
if we apply a linear-convergent
subproblem solver with at least a pre-specified number of iterations
to \cref{eq:quadratic}.
\ifdefined\arxiv
\else
\modify{In the following result, we introduce an operator $\Lambda_t\in \SS$
for computing generalized proximal steps in the subproblem, and
require that the eigenvalues of $\Lambda_t$ are upper-bounded by some
value fixed over all $t$.
This operator genralizes the case of taking a multiple of the identity
in the quadratic term of a proximal problem (applied in proximal
gradient as the step size), and can be treated as a sort of
preconditioner to allow for more possibilities of subproblem solvers.
We will also see after the proof of the following theorem some
examples of existing algorithms that can fit into this framework by
specifying different choices of $\Lambda_t$.}
\fi

\begin{theorem}
Consider the setting of \cref{thm:identify1}.
If \cref{alg:vm} is run with \cref{eq:subprob} and
\cref{eq:multiplicative} for some $\eta \in [0,1)$,
	\cref{assum:Hbound} holds,
and the update direction $p^{t}$ satisfies
\begin{equation}
	\label{eq:prox}
	x^t + p^{t} = \prox^{\Lambda_{t}}_{\Psi}\left(y^{t} -
	\Lambda_{t}^{-1} \left( \nabla f\left( x^t \right) + H_t \left( y^{t} -
	x^t \right) + s^{t} \right) \right),
\end{equation}
where $s^{t}$ satisfies $\|s^{t}\| \leq R\left( \norm{y^{t}
- (x^t + p^{t*})} \right)$ for some \modify{$R:[0,\infty) \rightarrow
[0,\infty)$ continuous in its domain} with $R(0) = 0$,
$\Lambda_{t} \in \SS$ with $M_1 \succeq \Lambda_{t}$ for $M_1 >
0$, and $y^{t}$ satisfies
\begin{align}
\label{eq:bdd0}
\norm{\left(y^{t} - x^t\right) - p^{t*}} &\leq \eta_1 \left(Q_t(0) -
Q_t^*\right)^\nu
\end{align}
for some $\nu > 0$ and $\eta_1 \geq 0$,
then there exists $\epsilon, \delta > 0$ such that $\|x^t - x^*\| \leq
\delta, |Q_t^*| \leq \epsilon$, and $\alpha_t = 1$ imply $x^{t+1} \in \M$.
\label{thm:identify2}
\end{theorem}

\begin{proof}
Suppose the statement is not true for contradiction.
Then there exist a continuous \modify{function $R: [0, \infty)
\rightarrow [0, \infty)$ with}
$R(0) = 0$, $\eta_1 \geq 0$, $M_1 > 0$, a sequence $\{x^t\} \subset
\H$ converging to $x^*$,
a sequence $\{H_t\} \subset \SS$ satisfying \cref{assum:Hbound} and
\begin{equation}
\lim_{t \rightarrow \infty}\, \min_{p}\quad Q_{H_t}\left( p;x^t
\right) = 0,
\label{eq:Qconv}
\end{equation}
three sequences $\{p^t\}, \{y^t\}, \{s^t\} \subset \H$
and a sequence $\{\Lambda_t\}\subset \SS$ with $M_1 \succeq \Lambda_t$
such that \cref{eq:subprob} with \cref{eq:multiplicative} and
\cref{eq:prox}-\cref{eq:bdd0} hold, yet $x^t + p^t \notin \M$
for all $t$.
We abuse the notation to let $p^{t*}$ and $Q^*_t$
respectively denote the optimal solution and objective value for
$\min_p Q_{H_t}(p;x^t)$, but $x^{t+1}$ is irrelevant to $p^t$ or
$\alpha_t$.

The optimality condition of \cref{eq:proximal} applied to
\cref{eq:prox} indicates that
\begin{equation}
	-\Lambda_{t}\left(x^{t} - y^{t}\right) - \left( \nabla
	f\left( x^t \right) + H_t \left( y^{t} - x^t \right) + s^{t}
	\right) \in \partial \Psi\left( x^{t} + p^t \right).
	\label{eq:opt}
\end{equation}
Thus, we have
\begin{align}
\nonumber
&~\dist\left( 0, \partial F\left( x^t + p^t \right) \right)\\
\nonumber
\leq &~\left\|\nabla f\left( x^t + p^t
	\right) -\Lambda_{t}\left(x^t + p^t -
	y^{t}\right) - \nabla f\left( x^{t} \right) \right.
	\left. -
	H_{t} \left( y^{t} - x^{t} \right) -
	s^{t} \right\|\\
\nonumber
\leq &~\norm{\nabla f\left( x^t + p^t \right)- \nabla
	f\left( x^{t} \right)} + M_1 \left( \norm{x^t - y^{t}} +
	\norm{p^t}\right)
	+ M
	\norm{y^{t} - x^{t}} + \norm{s^{t}}\\
\leq &~(L + M_1) \norm{p^t} +
	(M + M_1) \norm{y^{t} - x^{t}} +
	\norm{s^{t}}.
\label{eq:tobdd}
\end{align}
For the first two terms in \cref{eq:tobdd}, the
triangle inequality and \cref{eq:qg} imply
\begin{gather}
	\begin{aligned}
\norm{p^{t}} 
\leq \norm{p^{t} - p^{t*}} + \norm{p^{t*}}
\leq
&~\sqrt{2m^{-1}}\sqrt{ Q_{t}\left( p^{t} \right) -
Q^*_{t} } + \sqrt{2m^{-1}}\sqrt{ 0 - Q^*_{t}}\\
\stackrel{\cref{eq:multiplicative}}{\leq}&~ \left( \sqrt{\eta} + 1
\right)\sqrt{2m^{-1}}\sqrt{ - Q^*_{t}}, \text{ and }
\end{aligned}
\label{eq:bdd1}\\
\norm{y^{t}- x^{t}}
\leq \norm{\left(y^{t}- x^{t}\right) - p^{t*}} +
\norm{p^{t*}}
\stackrel{\cref{eq:bdd0}}{\leq}
\eta_1 \left( - Q^*_{t}\right)^\nu + \sqrt{2m^{-1}}\sqrt{ - Q^*_{t}}.
\label{eq:bdd3}
\end{gather}
For the last term in \cref{eq:tobdd}, we have from \modify{our definition of
$s^t$ that
\begin{align}
\norm{s^{t}} &\leq
R \left(\norm{y^{t} - \left(x^{t} + p^{t*}\right)}\right).
\label{eq:bdd4}
\end{align}
}
By substituting \cref{eq:bdd1}-\cref{eq:bdd4} back into
\cref{eq:tobdd}, clearly there are $C_1, C_2 > 0$ such that
\begin{equation}
\dist\left( 0, \partial F\left( x^t + p^t \right) \right)
\leq
C_1 \sqrt{ - Q^*_t} + C_2 \left( -Q^*_t \right)^\nu \modify{+
	R \left(\norm{y^{t} - \left(x^{t} + p^{t*}\right)}\right).}
\label{eq:tobdd2}
\end{equation}
Note that $Q_{H_t}(0;x^t) \equiv 0$, so $-Q^*_t \geq 0$ from its
optimality and the right-hand side of  \cref{eq:tobdd2} is
well-defined.
\modify{Next, we see from \cref{eq:Qconv}, \cref{eq:bdd0}, and the
continuity of $R$ that
\begin{equation}
\lim_{t \rightarrow \infty} \norm{\left(y^{t} - x^t\right) - p^{t*}} =
0, \quad \Rightarrow \quad
\lim_{t \rightarrow \infty} R\left( \norm{\left(y^{t} - x^t\right) -
p^{t*}}\right) =
0.
\label{eq:sconv}
\end{equation}
Applying \cref{eq:Qconv} and \cref{eq:sconv}} to \cref{eq:tobdd2} and
letting $t$ approach infinity then yield
\begin{equation}
\lim_{t \rightarrow \infty} \dist\left( 0, \partial F\left(
x^t + p^t \right) \right) = 0.
\label{eq:partialconv}
\end{equation}

Next, from \cref{eq:bdd1} and \cref{eq:Qconv}, it is also clear that
$\norm{p^t} \rightarrow 0$,
so from the convergence of $x^t$ to $x^*$ we have
\begin{equation}
x^t + p^t \rightarrow x^* + 0 = x^*.
\label{eq:convxp}
\end{equation}
Now \cref{eq:convxp} and \cref{eq:partialconv} allow us to apply
\cref{lemma:identify}
so $x^t + p^t \in \M$ for all $t$ large enough,
leading to the desired contradiction.
\qed
\end{proof}
\modify{The function $R$ can be seen as a general residual function and we
just need from it that $s^t$ approaches $0$ with $\norm{y^{t} - (x^t +
p^{t*}\modifyy{)}}$,
and \cref{thm:identify2} can be used as long as we can show that such
an $R$ exists, even if the exact form is unknown.} Condition
\cref{eq:bdd0} is deliberately chosen to exclude the objective
$Q_t(y^{t} - x^t)$ so that broader algorithmic choices like those with
$y^{t} \notin \dom(\Psi)$ can be included.

One concern for \cref{thm:identify2} is the requirement of $|Q_t^*|
\leq \epsilon$.
Fortunately, for \cref{alg:vm} with \cref{eq:subprob} and
\cref{eq:multiplicative}, if $\alpha_t$ are lower-bounded by some
$\bar \alpha > 0$ (which is true under \cref{assum:Hbound} by
\cite[Corollary~1]{LeeW18a}) and $F$ is lower-bounded,
then \cref{eq:armijo} together with \cref{eq:subprob} and
\cref{eq:multiplicative} shows that
$-Q_t^*$ is summable and thus decreasing to $0$.

We now provide several examples satisfying
\cref{eq:prox}-\cref{eq:bdd0} to demonstrate the usefulness of
\cref{thm:identify2}.
In our description below, $p^{t,i}$ denotes the $i$th iterate of
the subproblem solver at the $t$th outer iteration and $x^{t,i}\coloneqq
x^t + p^{t,i}$.
\begin{itemize}
	\item Using $x^{t+1}  = x^t + \bar p^t_\tau$ from \cref{eq:p}:
		Assume we have a tentative $p^t$ that
		satisfies \cref{eq:gm} for some $\hat \epsilon_t$, and we use
		\cref{eq:p} to generate $\bar p^t_\tau$ as the output
		satisfying \cref{eq:subprob} for a corresponding $\epsilon_t$
		calculated by \cref{eq:EB}.
		We see that it is of the form \cref{eq:prox} with $s^{t} =
		0$, $y^{t} = x^t + p^t$, and $\Lambda_{t} = I / \tau$.
		From \cref{eq:EB} we know that \cref{eq:subprob} is satisfied,
		while \cite[Corollary~3.6]{DruL16a} and \cref{eq:PG} guarantees
		\cref{eq:bdd0} with $\nu = 1/2$.
\item Proximal gradient (PG): These methods generate the inner
	iterates by
	\[
		x^{t,0} = x^t, \, x^{t,i+1} = \prox^{\lambda_{t,i}I}_{\Psi}\left( x^{t,i} -
		\lambda_{t,i}^{-1} \left( \nabla f\left( x^t \right) + H_t \left(
		x^{t,i} - x^t \right) \right) \right), \forall i > 0,
	\]
	for some $\{\lambda_{t,i}\}$ bounded in a positive range that
	guarantees $\{Q_t(p^{t,i})\}_{i}$ is a decreasing sequence
	for all $t$ (decided through pre-specification, line search, or
	other mechanisms).
	Therefore, for any $t$, no matter what value of $i$ is the last
	inner iteration, \cref{eq:prox} is satisfied with
		$y^{t} = x^{t,i-1}, s^{t} = 0, \Lambda_{t,i} =
		\lambda_{t,i-1} I$.
	The condition \cref{eq:multiplicative} holds for some $\eta < 1$
	because proximal-gradient-type methods are a descent method with
	$Q$-linear convergence on strongly convex problems, and
	\cref{eq:bdd0} holds for $\eta_1\sqrt{2 \eta / m}$ and $\nu = 1/2$
	from \cref{eq:qg}.
\item Accelerated proximal gradient (APG): The iterates are generated
	by
	\begin{gather}
		\begin{cases}
		y^{t,1} &= x^{t,0} = x^t,\\
		y^{t,i} &= x^{t,i-1} + \left( 1 - \frac{2}{\sqrt{\kappa(H_t)} +
		1} \right) \left(x^{t,i-1} - x^{t,i-2}\right), \forall i > 1,
	\end{cases}
		\label{eq:ag}\\
	x^{t,i} = \prox^{\norm{H_t} I}_{\Psi}\left( y^{t,i} -
	\norm{H_t}^{-1} \left( \nabla f\left( x^t \right) + H_t \left(
		y^{t,i} - x^t \right) \right) \right), \forall i > 1,
		\label{eq:ag2}
	\end{gather}
	where $\kappa(H_t) \geq 1$ is the condition number of $H_t$.
APG satisfies \citep{Nes18a}:
	\begin{equation}
		Q_t\left(x^{t,i} - x^t \right) - Q_t^* \leq -2\left( 1 -
		\kappa(H_t)^{-1/2} \right)^i Q_t^*,\quad \forall i >
		0, \quad \forall t \geq 0.
		\label{eq:agrate}
	\end{equation}
	Since $\kappa(H_t) \geq 1$,
	$p^{t,i}$ satisfies \cref{eq:subprob} with
	\cref{eq:multiplicative} for all $i \geq \ln 2\sqrt{\kappa(H_t)}$.
	If such a $p^{t,i}$ is output as our $p^t$, we see from
	\cref{eq:ag2} that \cref{eq:prox} holds with $s^t = 0$ and
	$\Lambda_{t} = \|H_t\| I$.
	The only condition to check is hence whether $y^{t,i}$ satisfies
	\cref{eq:bdd0}.
	The case of $i=1$ holds trivially with $\eta_1 = \sqrt{2/m}$.
	For $i > 1$, \cref{eq:bdd0} holds with $\eta_1 = 3\sqrt{2 \eta
		m^{-1}}$ and $\nu = 1/2$ because
	\begin{align*}
	&~\norm{y^{t,i} - \left( x^t + p^{t*} \right)}\\
	\stackrel{\cref{eq:ag}}{\leq} &~ \left( 1 -
	\frac{2}{\sqrt{\kappa(H_t)} + 1} \right) \norm{ x^{t,i-1} -
	x^{t,i-2}}+ \norm{x^{t,i-1} - \left( x^t + p^{t*} \right)} \\
	\stackrel{\cref{eq:qg},\cref{eq:agrate}}{\leq}&~
	\left(\norm{x^{t,i-1} -
	 x^t - p^{t*} } + \norm{x^{t,i-2} - x^t -
	p^{t*}}\right) + \sqrt{2\eta m^{-1}} \sqrt{-Q_t^*} \\
	\leq&~  3 \sqrt{2\eta m^{-1}} \sqrt{-Q_t^*}.
	\end{align*}
\item Prox-SAGA/SVRG:
	These methods update the iterates by
\begin{equation*}
	x^{t,i} = \prox^{\lambda_t I}_{\Psi}\left(x^{t,i-1} -
	\lambda_t^{-1} \left( \nabla f\left( x^t \right) + H_t \left(
	x^{t,i-1} - x^t \right) + s^{t} \right) \right),
\end{equation*}
with $x^{t,0} = x^t$,
$\{\lambda_t\}$ bounded in a positive range,
and $\{s^{t,i}\}$ are random variables converging to $0$ as $x^{t,i} -
x^t$ approaches $p^{t*}$.
(For a detailed description, see, for example, \cite{PooLS18a}.)
It is shown in \cite{XiaZ14a} that for prox-SVRG,
$Q_t(x^{t,i} - x^t) - Q_t^*$ converges linearly to
$0$ with respect to $i$ if $\lambda_t$ is small enough, so
\cref{eq:subprob} with \cref{eq:multiplicative} is satisfied.
A similar but more involved bound for prox-SAGA can
be derived from the results in \cite{DefBL14a}.
When $p^{t,i} = x^{t,i} - x^t$ for some $i > 0$ is output as $p^t$, we
get $y^{t} = x^{t,i-1}$, $\Lambda_{t} = \lambda_t I$, and $s^t =
s^{t,i}$ in \cref{eq:prox}, so the requirements of \cref{thm:identify2}
hold.
\end{itemize}
If $\Psi$ is block-separable and $\M$ decomposable into a product
of submanifolds that conform with the blocks of $\Psi$,
\cref{eq:prox} can be modified easily to suit block-wise solvers like
block-coordinate descent.
This extension simply adapts the analysis above to the
block-wise setting, so the proof is straightforward and omitted for
brevity.

\subsection{Iterate Convergence}
\cref{thm:identify1} and \cref{thm:identify2} both indicate that for
\cref{alg:vm} to identify $\M$, we need $x^t$ (or at least a
subsequence) to converge to a point $x^*$ of partial smoothness.
We thus complement our analysis to show the convergence of the
iterates under convexity of $f$ and a local sharpness condition, which
is a special case of the KL condition and is universal in real-world
problems, without any additional requirement on the algorithm such as the
relative-error condition in \cite{AttBRS10a,BonLPPR17a,BonPR20a}.
In particular, we assume $F$ satisfy the following for some $\zeta,
\xi > 0$, and $\theta \in (0,1]$:
\begin{equation}
\zeta \dist(x,\Omega) \leq \left(F(x) - F^*\right)^{\theta},\quad
\forall x \in \levbet.
\label{eq:sharpness}
\end{equation}
This becomes the well-known quadratic growth condition when $\theta = 1/2$,
and $\theta = 1$ corresponds to the weak-sharp minima \citep{BurF93a}.
As discussed in \cref{sec:related}, under convexity of $f$,
\cite{MCY19a} showed that the iterates of their PN variant converge to
some $x^* \in \Omega$ if $F$ satisfies EB,
which is equivalent to the quadratic growth condition in their setting
\citep{DruL16a}.
Our analysis allows broader choices of $H_t$ and \modify{(strong)} iterate
convergence is proven for $\theta \in (1/4,1]$.
\begin{theorem}
Consider \cref{alg:vm} with any $x^0 \in \H$.
Assume that $\Omega \neq \emptyset$, $f$ is convex and
$L$-smooth for $L > 0$, $\Psi \in \Gamma_0$,
\cref{assum:Hbound} holds, there is $\eta \in [0,1)$
such that $p^t$ satisfies \cref{eq:subprob} with
\cref{eq:multiplicative} for all $t$,
and that \cref{eq:sharpness} holds for $\xi, \zeta > 0$ and
$\theta \in (1/4,1]$. Then $x^t \rightarrow x^*$ for some $x^* \in
\Omega$.
\label{thm:iterate}
\end{theorem}
The convergence in \cref{thm:iterate} holds true in infinite
dimensional real Hilbert spaces \modify{with strong convergence (which is
indistinguishable from weak convergence in the finite-dimensional
case)}, and the proof in \cref{app:proof} is written in this general
scenario.
The key of its proof is our following improved convergence rate,
which might have its own interest.
Except for that the case of $\theta = 1/2$ has been proven by
\citet{PenZZ18a} \modify{and that $\theta = 0$ reduces to the general convex
	case analyzed in \cite{LeeW18a}}, this faster convergence rate is, up
to our knowledge, new for \ipvm.
\begin{theorem}
Consider the settings of \cref{thm:iterate} but with $\theta \in
[0,1]$, $\H$ a real Hilbert space and $x^0 \in \levbet$.
Then there is $\bar{\alpha} > 0$ such that $\alpha_t \geq \bar \alpha$
for all $t$ and the following hold.
\begin{enumerate}
\item For $\theta \in (1/2,1]$:
	When $\delta_t$ is large enough such that
	\begin{equation}
		\delta_t > \left( \zeta^2M^{-1} \right)^{\frac{1}{2\theta
		- 1}},
	\label{eq:bound}
	\end{equation}
	we have
	\begin{align}
		\delta_{t+1} &\leq \delta_t \left( 1 -
		\frac{(1 - \eta)\alpha_t \gamma \zeta^2 \xi^{1 - 2\theta}}{2M}
		\right).
		\label{eq:earlylinear}
	\end{align}
	Next, let $t_0$ be the first index failing \cref{eq:bound},
	then for all $t \geq t_0$ we have
	\begin{equation}
		\delta_{t+1} \leq \delta_t \left( 1 - \frac{(1 - \eta)
		\alpha_t \gamma}{2} \right).
		\label{eq:weaksharp}
	\end{equation}
\item For $\theta = 1/2$, we have global Q-linear convergence of
	$\delta_t$ in the form
\begin{equation}
	\frac{\delta_{t+1}}{\delta_t}\leq 1 - \left(1 - \eta\right)\alpha_t
	\gamma \cdot
\begin{cases}
	\frac{\zeta^2}{2 \|H_t\|}, &\text{ if } \zeta^2 \leq \|H_t\|,\\
	\left(1 - \frac{\|H_t\|}{2 \zeta^2}\right), &\text{ else,}
\end{cases}
\quad \forall t \geq 0.
\label{eq:qlinear}
\end{equation}
\item For $\theta \in \modify{[}0,1/2)$, \cref{eq:weaksharp} takes place when
	$\delta_t$ is large enough to satisfy \cref{eq:bound}.
	Let $t_0$ be the first index such that \cref{eq:bound} fails,
	then
	\begin{align}
		\frac{\delta_t}{ \delta_{t_0}} \leq
	\left( \left(1 - 2\theta\right) \sum_{i=t_0}^{t-1} \alpha_i
	\right)^{-\frac{1}{1-2\theta}}
	\leq 
	\left( \left(1 - 2\theta\right) (t - t_0)\bar \alpha
	\right)^{-\frac{1}{1-2\theta}},
	\,\forall t \geq t_0.
	\label{eq:sublinear}
	\end{align}
\end{enumerate}
\label{thm:rate}
\end{theorem}

\modifyy{For the range of $\theta$ in Theorem~\ref{thm:iterate}, convergence of
the iterates generated by inexact SQA is a new result.
Moreover, even if the additional conditions on the forward-backward
envelope in \cite{BonPR20a} also hold, although their analysis also uses
the subproblem stopping condition \cref{eq:subprob},
they need a much stricter stopping tolerance of the form
\[
	\epsilon_t = - \tau_t  Q_t^*,\quad \sum\sqrt{\tau_t}
< \infty,
\]
which clealy requires $\tau_t$ to converge to $0$ fast enough and thus
costs much more time in the subproblem solve.
In contrast, we just need $\epsilon_t$ to be a constant factor of
$Q_t^*$ in \cref{eq:multiplicative}, so the number of inner iterations
and thus the cost of subproblem solve can be a constant.}

\section{An Efficient Inexact SQA Method with
Superlinear Convergence in Running Time}
\label{sec:alg}
Now that it is clear \ipvm is able to identify the active manifold,
we utilize the fact that the optimization problem reduces to a smooth
one after the manifold is identified to devise more efficient
approaches, with safeguards to ensure that the correct manifold is
really identified.
The improved algorithm, \ipn, is presented in \cref{alg:vm2} and we
explain the details below.

\begin{algorithm}
\DontPrintSemicolon
\SetKwInOut{Input}{input}\SetKwInOut{Output}{output}
\caption{\ipn: An improved inexact successive quadratic approximation method
utilizing manifold identification}
\label{alg:vm2}
\Input{$x^0 \in \H$, $\gamma, \beta, \in (0,1)$,
$S,T \in \N$, a subproblem solver $\mathcal{A}$ satisfying
\cref{eq:prox}-\cref{eq:bdd0} and is linear convergent for
\cref{eq:quadratic}}

Compute an upper bound \modify{$\hat L$} of the Lipschitz constant $L$
\modify{for $\nabla f$}

	$\text{SmoothStep} \leftarrow 0$, $\text{Unchanged} \leftarrow 0$

\For{$t=0,1,\dotsc$}{
	Identify $\M \ni x^t$ such that $\Psi$ is partly smooth
	relative to $\M$ at $x^t$

	\If{$\M$ remains the same from last iteration and $\M \neq
	\emptyset$}
	{
		$\text{Unchanged} \leftarrow
		\text{Unchanged} + 1$
	}

	\uIf{\modify{$\text{Unchanged} < S$}}
	{
		$\text{SmoothStep} \leftarrow 0$

		Decide $H_t$ and solve \cref{eq:quadratic} using $\mathcal{A}$
		with at least $T$ iterations

		\While{\cref{eq:armijo} is not satisfied with $\alpha_t = 1$}
		{Enlarge $H_t$ and resolve \cref{eq:quadratic} using
		$\mathcal{A}$ with at least $T$ iterations}

		$x^{t+1} \leftarrow x^t + p^t$
	}
	\uElse
	{
		\uIf{$\text{SmoothStep} = 1$}
		{
			$\text{SmoothStep} \leftarrow 0$ and
			conduct a proximal gradient step
			\begin{equation}
				x^{t+1} = \prox^{\modify{\hat L} I}_{\Psi}\left( x^t -
				\modify{\hat L}^{-1} \nabla f\left( x^t \right) \right)
			\label{eq:proxgrad1}
			\end{equation}

		}
		\uElse{
			{
			Manifold optimization: try to find $x^{t+1}$ with
			$F(x^{t+1}\modify{)} \leq F(x^t)$ by
			\begin{itemize}
			\item {\bf Variant I:} truncated Newton: \cref{alg:newton}
			\item {\bf Variant II:} Riemannian quasi-Newton
			\end{itemize}
		}

			\lIf{Manifold optimization fails}
			{
				$x^{t+1} \leftarrow x^t$, $\text{Unchanged} \leftarrow
				0$
			}
			\lElse{
			SmoothStep $\leftarrow 1$
		}
		}
	}
}
\end{algorithm}

\ipn has two stages, separated by the event of identifying the active
manifold $\M$ of a cluster point $x^*$.
Our analysis showed that iterates converging to $x^*$ will eventually
identify $\M$, but since neither $x^*$ nor $\M$ is known a priori,
the conjecture of identification can only be made when $\M$ remains
unchanged for $S > 0$ iterations.

Most parts in the first stage are the same as \cref{alg:vm},
although we have added specifications for the subproblem solver
according to \cref{thm:identify2}.
The only major difference is that instead of linesearch,
\ipn adjusts $H_t$ and re-solve \cref{eq:quadratic} whenever
\cref{eq:armijo} with $\alpha_t = 1$ fails.
This trust-region-like approach has guaranteed global convergence from
\cite{LeeW18a} and ensures $\alpha_t = 1$ for \cref{thm:identify2} to
be applicable.

In the second stage, we alternate between a standard proximal gradient
(PG) step \cref{eq:proxgrad1} and a manifold optimization (MO) step.
PG is equivalent to solving \cref{eq:quadratic} with $H_t =
\modify{\hat L} I$ to
optimality, so \modify{\cref{thm:identify2}} applies.
When $\M$ is not correctly identified, a PG step thus prevents us from
sticking at a wrong manifold, while when the superlinear convergence
phase of the MO step is reached, using PG instead of solving
\cref{eq:quadratic} with a sophisticated $H_t$ avoids redundant
computation.

When the objective is partly smooth relative to a manifold $\M$,
optimizing it within $\M$ can be cast as a manifold optimization
problem, and efficient algorithms for this type is abundant (see, for
example, \cite{AbsMS09a} for an overview).
The difference between applying MO methods and
sticking to \cref{eq:quadratic} is that in the former, we can obtain
the exact solution to the subproblem for generating the update
direction \modify{in finite time, because the subproblem in the MO
step is simply an unconstrained quadratic optimization problem whose
solution can be found by solving a linear system,}
while in the latter it takes indefinitely long to compute the exact
solution, so the former is preferred in practice for better running
time.
Although we did not assume that $f$ is twice-differentiable, its
generalized Hessian (denoted by $\nabla^2 f$) exists everywhere
since the gradient is Lipschitz-continuous \citep{HirSN84a}.
As discussed in \cref{sec:background}, we can find a $\mathcal{C}^2$
parameterization $\phi$ of $\M$ around $x^*$, and we use this $\phi$
to describe a truncated semismooth Newton (TSSN) approach.
Since $\M$ might change between iterations, when we are conducting MO
at the $t$th iteration, we find a parameterization $\phi_t$ of the
current $\M$ and a point $y^t$ such that $\phi_t(y^t) = x^t$.
If $\M$ remains fixed, we also retain the same $\phi_t$.
The TSSN step $q^t$ is then obtained by using preconditioned CG (PCG,
see for example \cite[Chapter~5]{NocW06a}) to find an approximate
solution for
\begin{equation}
	\modify{
	\begin{aligned}
	q^t &\approx \argmin_{q}\quad \inprod{g^t}{q} + \frac12
	\inprod{q}{H_t q}, \quad \text{ or equivalently } \quad
H_t q^t \approx -g^t,\\
g^t &\coloneqq \nabla F(\phi_t(y^t)), \;
H_t  \coloneqq \nabla^2 F\left( \phi_t(y^t) \right) + \mu_t I, \;
\mu_t \coloneqq c \norm{g^t}^{\rho},
\end{aligned}
}
\label{eq:newton}
\end{equation}
that satisfies
\begin{equation}
	\label{eq:stop}
\norm{H_t q^t + g^t} \leq 0.1
\min\left\{\norm{g^t}, \norm{g^t}^{1 + \rho} \right\}
\end{equation}
with pre-specified $c > 0$ and $\rho \in (0,1]$.
We then run a backtracking line search procedure to find a suitable
step size $\alpha_t > 0$.
For achieving superlinear convergence, we should accept unit step size
whenever possible, so we only require the objective not to increase.
If $q^t$ is not a descent direction or $\alpha_t$ is too small,
we consider the MO step failed and go back to the first stage.
If $\alpha_t < 1$, the superlinear convergence phase is not entered
yet, and likely $\M$ has not been correctly identified, so we also
switch back to the first stage.
This algorithm is summarized in \cref{alg:newton}.
When products between $\nabla^2 F(\phi_t(y^t))$ and arbitrary
points, required by PCG, cannot be done easily, one can adopt Riemannian
quasi-Newton approaches like \cite{HuaGA15a} instead.

\ifdefined\arxiv
\else
\modify{
Viewing from \cref{eq:newton}, it is clear that our previous claim
that the exact solution of the subproblem can be found in finite time
\modifyy{holds}.
In particular, the second form of the subproblem indicates that the
exact solution can be obtained by solving a linear system, and this
can be done in time at most cubic to the manifold dimension
by first finding a matrix representation for $H_t$ and then inverting
it.
If there is no solution to the subproblem, such a numerical linear
algebra approach can also detect the infeasibility of the linear
system \modifyy{with} the same time cost.
Similarly, Riemannian quasi-Newton approaches also have the same time
cost upper bound for computing the update direction.
The PCG procedure we adopted also has such a finite-termination
guarantee with the same cost bound, and it comes additionally
with a guarantee that the objective value (the first form in
\cref{eq:newton}) converges to the optimum at the same R-linear rate
as the accelerated gradient method.
However, in practice we often observe that PCG converges much faster
than this worst-case convergence speed guarantee and can satisfy
\cref{eq:stop} with time cost significantly less than either applying
a (proximal) first-order method to \cref{eq:quadratic} or obtaining an
exact solution of \cref{eq:newton} by resorting to matrix
decompositions.

The parameter $S$ for thresholding the entry of the second stage
indicates our confidence in the currently identified manifold.
If we think the identification is reliable, we can set it to a lower
value so that we enter the manifold optimization stage earlier.
In general, we would like to set $S$ in an appropriate range so that the
algorithm does not spend too much time on a wrong manifold, and when
the right manifold is found, it does not stay in the first stage long
either.
In practice, more sophisticated strategies to adaptively change $S$
might lead to better efficiency.
For example, we can start with a small $S$ and restrict the running
time of the MO step to be proportional to $S$ steps of the first stage.
When unit step size is not accepted in an MO step, or when the
manifold changes in a proximal gradient step, indicating that the fast
local superlinear convergence is not yet achieved, we can switch back
to the first stage and increase $S$.
}
\fi

\begin{algorithm}[tb]
\DontPrintSemicolon
\SetKwInOut{Input}{input}\SetKwInOut{Output}{output}
\caption{Truncated semismooth Newton on Manifold}
\label{alg:newton}
\Input{$x^t \in \H$, $c > 0$, $\rho \in (0,1]$, $\underline \alpha,
\beta, \gamma \in (0,1)$, and a manifold $\M$}
Obtain parameterization $\phi_t$ for $\M$ and point $y^t$ with $\phi_t(y^t) = x^t$

$\alpha_t \leftarrow 1$, and compute $g^t$ in \cref{eq:newton}

Obtain by PCG an approximation solution $q^t$ to \cref{eq:newton}
satisfying \cref{eq:stop}

\lIf{$\inprod{q^t}{g^t} \geq 0$}
{Report MO step fails and exit}

\lWhile{$F(\phi_t(y^t + \alpha_t q^t)) > F(x^t)$ and $\alpha_t >
\underline \alpha$}
{$\alpha_t \leftarrow \beta \alpha_t$}

\lIf{$\alpha_t \leq \underline \alpha$}
{Report MO step fails}
\lElse{
	$x^{t+1} \leftarrow \phi_t(y^t + \alpha_t q^t)$}

\end{algorithm}

\modify{
\subsection{Global Convergence}
\label{subsec:analysis}
This section provides global convergence guarantees for \cref{alg:vm2}.
Because MO steps do not increase the objective value,
global convergence of \cref{alg:vm2} follows from the analysis in
\cite{LeeW18a} by treating \cref{eq:proxgrad1} as
solving \cref{eq:quadratic} with $H_t = \hat L I$ and noting that this
update always satisfies \cref{eq:armijo} with $\alpha_t = 1$.
For completeness, we still state these results, and provide a proof in
\cref{app:proof2}.

First, we restate a result in \cite{LeeW18a} to bound the number of
steps spent in the while-loop for enlarging $H_t$ in \cref{alg:vm2}.
\begin{lemma}[{\cite[Lemma~4]{LeeW18a}}]
\label{lemma:Hbound}
Given an initial choice $H_t^0$ for $H_t$ at the \modifyy{$\mathrm{t}$th} iteration of
\cref{alg:vm2} (so initially we start with $H_t = H_t^0$ and modify it
when \cref{eq:armijo} fails with $\alpha_t = 1$) and a parameter
$\beta \in (0,1)$\modifyy{. C}onsider the following two variants for enlarging
$H_t$, starting with $\sigma = 1$:
\begin{align}
	\sigma \leftarrow \beta \sigma, \quad
	H_t \leftarrow H_t^0 / \sigma,
	\label{eq:enlarge1}
	\tag{Variant 1}\\
	H_t \leftarrow H_t^0 + \sigma^{-1} I,\quad \sigma \leftarrow \beta
	\sigma\modifyy{.}
	\label{eq:enlarge2}
	\tag{Variant 2}
\end{align}
\modifyy{W}e then have the following bounds if every time the approximate
solution to \cref{eq:quadratic} always satisfies $Q_{H_t}(p^t;x^t)
\leq 0$.
\begin{enumerate}
\item If $H_t^0$ satisfies
	$M_0^t \succeq H_t^0 \succeq m_0^t$
for some $M_0^t \geq m_0^t > 0$, then the final $H_t$ from \cref{eq:enlarge1} satisfies
$\|H_t\| \leq M \max\left\{1, L / (\beta m)\right\}$,
and the while-loop terminates within $\lceil \log_{\beta^{-1}} L/m \rceil$ rounds.

\item If $H_t^0 \succeq 0$, then the final $H_t$ from
	\cref{eq:enlarge2} satisfies
$\norm{H_t} \leq M + \max\left\{1, L / \beta \right\}$
and the while-loops terminates within $1 + \lceil \log_{\beta^{-1}}{L} \rceil$ rounds.
\end{enumerate}
\end{lemma}

Now we provide global convergence guarantees for \cref{alg:vm2}
without the need of manifold identification.
From \cref{lemma:Hbound}, we can simply assume without loss of
generality that \cref{assum:Hbound} holds for the final $H_t$ that
leads to the final update direction that satisfies \cref{eq:armijo}.

\begin{theorem}
\label{thm:rate2}
Consider \cref{eq:f} with $f$ $L$-smooth for $L > 0$, $\Psi \in
\Gamma_0$, and $\Omega \neq \emptyset$.
Assume \cref{alg:vm2} is applied with an initial point $x^0$, the estimate
$\hat L$ satisfies $\hat L \geq L$,
\cref{assum:Hbound} holds for the final $H_t$ after exiting the
while-loop,
and \cref{eq:subprob} is satisfied with \cref{eq:multiplicative} for
some $\eta \in [0,1)$ fixed over $t$.
Let $\{k_t\}$ be the iterations that the MO step is not
attempted (so either \cref{eq:quadratic} is solved  approximately
or \cref{eq:proxgrad1} is conducted),
then we have $k_t \leq 2t$ for all $t$.
By denoting $\tilde M \coloneqq \max\left\{ \hat L, M \right\}, \quad
\tilde m \coloneqq \min\left\{ \hat L, m \right\}$,
we have the following convergence rate guarantees.
\begin{enumerate}
\item Let $G_t \coloneqq \argmin_p Q_I(p;x^t)$, then $\norm{G_{k_t}}
	\rightarrow 0$, and for all $t\geq 0$, we have
\[
	\min_{0 \leq i \leq t} \norm{G_{k_i}}^2 \leq \frac{F(x^0) -
	F^*}{\gamma (t+1)} \frac{\tilde M^2 \left( 1 + \tilde m^{-1} +
	\sqrt{1 - 2 \tilde M^{-1} + \tilde m^{-2}}\right)^2}{2(1 - \eta)
	\tilde m}.
\]
Moreover, $G_t = 0$ if and only if $0 \in \partial F(x^t)$, and
therefore any limit point of $\{x^{k_t}\}$ is a stationary point of
\cref{eq:f}.

\item
	If in addition $f$ is convex and there exists $R_0 \in [0,
		\infty)$ such that
\begin{equation}
\sup_{x: F\left( x \right) \leq F\left( x^0
	\right)}\left\|x - P_\Omega(x)\right\| = R_0
	\label{eq:R0}
\end{equation}
\modifyy{(}in other words, \cref{eq:sharpness} holds with $\theta = 0$, $\zeta =
R_0^{-1}$, and $\xi = F(x^0) - F^*$\modifyy{),} then:
\begin{enumerate}
	\item When $F(x^{k_t}) - F^* \geq \tilde M R_0^2$, we have
		\[
			F\left( x^{k_t+1} \right) - F^* \leq  \left( 1 - \frac{	\gamma(1
			- \eta)	}{2} \right) \left( F\left( x^{k_t} \right)
			- F^*\right).
		\]
	\item Let $t_0 \coloneqq \argmin_t \{t: F\left( x^{k_t} \right)
	- F^* < \tilde M R_0^2\}$, we have for all $t \geq t_0$ that
	\[
		F\left( x^{k_t} \right) - F^* \leq \frac{2 \tilde M R_0^2}{
		\gamma \left( 1 - \eta \right) \left( t - t_0 \right) + 2}.
	\]
	Moreover, we have
	\[
		t_0 \leq \max \left\{ 0, 1 + \frac{2}{\gamma \left( 1 - \eta
		\right)} \log \frac{F\left( x^0 \right) - F^*}{\tilde M
		R_0^2}\right\}.
	\]
\end{enumerate}
In summary, we have $F(x^t) - F^* = O\left( t^{-1} \right)$.
\item The results of \cref{thm:rate} hold, with $M$ and $\norm{H_t}$
	replaced by $\tilde M$, $\delta_t$ by $\delta_{k_t}$, $\alpha_t$
	and $\bar \alpha$ by $1$, $\delta_t$ by $\delta_{k_t}$,
	$\delta_{t+1}$ by $\delta_{k_t+1}$, and $\delta_{t_0}$ by
	$\delta_{k_{t_0}}$.
\end{enumerate}
\end{theorem}
}

\subsection{Superlinear and Quadratic Convergence}
\label{subsec:super}
\modify{
Following the argument in the previous subsection to treat
\cref{eq:proxgrad1} as solving \cref{eq:quadratic} exactly,
the manifold identification property of \cref{eq:proxgrad1} also
follows from \cref{thm:identify2}.}
We thus focus on its local convergence in this subsection.
In what follows, we will show that $x^t$ converges to a stationary point
$x^*$ satisfying \cref{eq:nod} superlinearly or even quadratically in
the second stage.

Let $\M$ be the active manifold of $x^*$ and $\phi$ be a
parameterization of $\M$ with $\phi(y^*) = x^*$ for some point $y^*$.
We can thus assume without loss of generality $\phi_t = \phi$ for all
$t$ that identified $\M$.
We denote $F_{\phi} (y) \coloneqq F(\phi(y))$.
\modifyy{For simplicity, we assume that $F_{\phi}$ is twice-differentiable with
its Hessian locally Lipschitz continuous around $y^*$.
In particular, we just need the following property to hold locally in
a neighborhood $U_0$ of $y^*$}:
\begin{equation}
\nabla F_\phi(y_1) - \nabla F_\phi(y_2) - \nabla^2
	F_\phi(y_2) = O\left( \|y_1 - y_2\|^2 \right),
	\quad \forall y_1, y_2 \in U_0.
\label{eq:semismoothp}
\end{equation}
We do not assume $\nabla^2 F_\phi(y^*) \succ 0$ like
existing analyses for Newton's method, but consider a
degenerate case in which there is a neighborhood $U_1$ of $y^*$ such
that
\begin{equation}
	\nabla^2 F_\phi(y) \modify{\succeq} 0, %
	\quad \forall y \in U_1.
\label{eq:psd}
\end{equation}
Note that \cref{eq:psd} implies that $F_\phi$ is convex within
$U_1$.
We can decompose $\H$ into the direct sum of the tangent and the
normal spaces of $\M$ at $x^*$, and thus its stationarity implies
$\nabla F_{\phi}(y^*) = 0$.
This and \cref{eq:psd} mean $y^*$ is a local optimum of $F_{\phi}$,
and hence $x^*$ is a local minimum of $F$ when $f$ is $L$-smooth,
following the argument of \cite[Theorem~2.5]{Wri12a}.
We also assume that $F_\phi$ satisfies a sharpness condition similar
to \cref{eq:sharpness} in a neighborhood $U_2$ of $y^*$:
\begin{equation}
\zeta^{\hat \theta} \norm{y-y^*} \leq \left(F_\phi(y) - F\left( y^*
\right)\right)^{\hat\theta},\quad \forall y \in U_2,
\label{eq:sharpness2}
\end{equation}
for some $\zeta > 0$ and $\modify{\hat \theta} \in
(0,1/2]$.\footnote{$\modify{\hat
\theta} > 1/2$ cannot happen unless $F_{\phi}$ is a constant in
$U_2$.}
By shrinking the neighborhoods if necessary, we assume without
loss of generality that $\modifyy{U_0 =} U_1 = U_2$ and denote it by $U$.
\modifyy{Note that the conventional assumption of %
positive-definite Hessian at $y^*$ is a special case that satisfies
\cref{eq:psd} and \cref{eq:sharpness2} with $\hat \theta = 1/2$.}

We define $d_t \coloneqq \norm{y^t - y^*}$ and use it to bound
$y^t + q^t$ and $\nabla F_\phi(y^t + q^t)$.
\begin{lemma}
Consider a \modify{stationary} point $x^*$ of \cref{eq:f} with $\Psi$ partly
smooth at it relative to a manifold $\M$ with a parameterization
$\phi$ and a point $y^*$ such that $\phi(y^*) = x^*$,
\modifyy{and assume that within a neighborhood $U$ of $y^*$, $F_{\psi}$ is
twice-differentiable with \cref{eq:psd,eq:semismoothp} hold.}
Then $y^t \in U$ implies that any $q^t$ satisfying \cref{eq:stop}
is bounded by
\begin{align}
\norm{q^t}
\leq  2 d_t + \mu_t^{-1} O\left( d_t^{2} \right) + 0.1 \mu_t^{-1}
	\norm{g^t}^{1 + \rho}.
\label{eq:q}
\end{align}
\label{lemma:qt}
\end{lemma}
\begin{proof}
From \cref{eq:stop}, we can find $\psi_t \in \H$ such that
\begin{equation}
H_t q^t + g^t = \psi_t, \quad \norm{\psi_t} \leq 0.1
\norm{g^t}^{1 + \rho}.
\label{eq:error}
\end{equation}
From \cref{eq:psd} and \cref{eq:newton}, we have
\begin{equation}
H_t \succeq \mu_t \succ 0,
\label{eq:Hlower}
\end{equation}
so $H_t$ is invertible.
We then get
\begin{align}
\nonumber
\norm{y^t + q^t - y^*}
\stackrel{\cref{eq:error}}{=}&~ \norm{H_t^{-1} \left( \psi_t - g^t +
	H_t\left( y^t - y^* \right) \right)}\\
\nonumber
\stackrel{\cref{eq:newton}}{\leq} &~ \norm{H_t^{-1}} \left(\norm{\psi_t} +
	\norm{g^t - \nabla^2 F\left( \phi_t\left( y^t \right)
	\right) d_t } + \mu_t d_t\right)\\
\stackrel{\cref{eq:Hlower},\cref{eq:error},\cref{eq:semismoothp}}{\leq}&~ 0.1 \mu_t^{-1}
	\norm{g^t}^{1 + \rho} + \mu_t^{-1} O\left( d_t^{2} \right) + d_t.
\label{eq:bd1}
\end{align}
From the triangle inequality, we have
$\norm{q^t} \leq \norm{y^t - y^*} + \norm{y^t + q^t - y^*}$,
whose combination with \cref{eq:bd1} proves \cref{eq:q}.
\qed
\end{proof}
\begin{lemma}
Consider the setting of \cref{lemma:qt} and further assume that
$\Psi \in \Gamma_0$ and $f$ is $L$-smooth.
The following hold.
\begin{enumerate}
\item If $\rho \in (0,1]$ and $F_\phi$ satisfies
	\cref{eq:sharpness2} with $\hat \theta = 1/2$ for some $\zeta > 0$,
	then
\begin{equation}
	\norm{ y^t + q^t - y^*} = O\left( d_t^{1+\rho}
\right), \quad
\norm{\nabla F_\phi (y^t + q^t)} = O\left( \norm{g^t}^{1+\rho} \right).
\label{eq:quadconv}
\end{equation}
\item If $\rho = 0.69$ and $F_\phi$ satisfies
	\cref{eq:sharpness2} for some $\zeta > 0$ and $\hat \theta \geq 3/8$,
	then
\begin{equation}
	\norm{ y^t + q^t - y^*} = o\left( d_t \right),
\quad
\norm{\nabla F_\phi (y^t + q^t)} = o\left( \norm{g^t} \right).
\label{eq:superlinear}
\end{equation}
\end{enumerate}
\label{lemma:yt}
\end{lemma}

\begin{proof}
From the assumptions on $\Psi$, $\phi$, and $f$, $F_\phi$
is twice-differentiable almost everywhere,
and within any compact set $K$ containing $y^*$, any
$\nabla^2 F_\phi \in \partial \nabla F_\phi$ is upper-bounded by some $L_K > 0$
($f(\phi(y))$ is \modify{differentiable with the gradient Lipschitz-continuous}, and $\nabla^2
(\Psi(
\phi(y))$ is upper-bounded), so $ F_\phi$ is $L_K$-smooth within $K$.
Since $K$ is arbitrary, we let $K \supset U$ and obtain
\begin{equation}
\norm{g^t} \leq L_K d_t.
\label{eq:gt}
\end{equation}
Since \cref{eq:psd} implies that $F_\phi$ is convex in $U$,
\cref{eq:sharpness2} leads to
\begin{equation}
	\norm{\nabla F_{\phi} \left( y \right)} \geq \zeta \norm{y - y^*}
	^{(1 - \hat \theta)/\hat \theta}, \quad
	\forall y \in U.
\label{eq:gradbound}
\end{equation}

For the first case, \cref{eq:gt}-\cref{eq:gradbound} show $g^t =
\Theta(d_t)$, so \cref{lemma:qt} implies
\begin{equation}
\norm{q^t} = O(d_t) = O(g^t).
\label{eq:qtbound}
\end{equation}
Thus, by the triangle inequality, \cref{eq:quadconv} is proven by
\begin{align}
\nonumber
\norm{\nabla F_\phi\left(y^t + q^t\right)}
= &~\norm{\nabla F_\phi\left(y^t + q^t\right) + \psi_t - \psi_t}\\
\nonumber
\stackrel{\cref{eq:error}}{\leq} &~ \norm{\nabla F_\phi\left(y^t +
	q^t\right) - g^t - H_t q^t} + 0.1 \norm{g^t}^{1+\rho} \\
\stackrel{\cref{eq:semismoothp},\cref{eq:newton}}{\leq} &~ O\left(
	\norm{q^t}^2 \right) + \mu_t \norm{q^t} + 0.1 \norm{g^t}^{1+\rho}
\stackrel{\cref{eq:qtbound},\cref{eq:newton}}{=}  O\left(
\norm{g^t}^{1+\rho} \right).
\label{eq:yq}
\end{align}

In the second case,
$\hat \theta/(1 - \hat \theta) \geq 3/5$, so
\cref{eq:newton}, \cref{lemma:qt} and \cref{eq:gt}-\cref{eq:gradbound} imply
\begin{equation}
\norm{q^t} = O\left(d_t^{0.85}\right).
\label{eq:qbound}
\end{equation}
We then get from $\rho = 0.69$ that
\begin{align}
\nonumber
\norm{\nabla F_\phi\left(y^t + q^t\right)}
\stackrel{\cref{eq:yq}}{\leq}
&~ O\left(
	\norm{q^t}^2 \right) + \mu_t \norm{q^t} + 0.1 \norm{g^t}^{1.69}\\
\stackrel{\cref{eq:gt},\cref{eq:newton},\cref{eq:qbound}}{=} &~
	O\left( d_t^{1.7} \right) + 2\mu d_t + O(d_t^2) + O\left(
	d_t^{1.69}\right)
= O(d_t^{1.69}).
\label{eq:gbound}
\end{align}
From \cref{eq:gradbound} we get that
\begin{equation*}
\norm{y^t + q^t - y^*} = O\left(
\norm{\nabla F_{\phi} \left( y^t + q^t \right)}^{0.6}\right)
\stackrel{\cref{eq:gbound}}{\leq} O\left( (d_t^{1.69})^{0.6} \right) =
O\left( d_t^{1.014} \right),
\end{equation*}
proving the first equation in \cref{eq:superlinear}.
The second one is then proven by
\begin{equation*}
	\norm{\nabla F_{\phi} \left( y^t + q^t \right)}
\stackrel{\cref{eq:gbound}}{\leq} O\left( (d_t^{1.69})\right)
\stackrel{\cref{eq:gradbound}}{=} O\left( \norm{g^t}^{1.014} \right).
\tag*{\qed}
\end{equation*}
\end{proof}

Now we are able to show two-step superlinear convergence of $\norm{x -
x^*}$.
\begin{theorem}
Consider the setting of \cref{lemma:yt} and assume in addition that
$x^*$ satisfies \cref{eq:nod}.
Then there is a neighborhood $V$ of $x^*$ such that if at the $t_0$th
iteration of \cref{alg:vm2} for some $t_0 > 0$ we have that $x^{t_0} \in V$,
Unchanged $\geq S$, $\M$ is correctly identified with
parameterization $\phi$ and $\phi(y^*) = x^*$, and $\alpha_t =
1$ is taken in \cref{alg:newton} for all $t \geq t_0$, we get the
following for all $t \ge t_0$.
\begin{enumerate}
\item For $\rho \in (0,1]$ and $F_\phi$ satisfying
	\cref{eq:sharpness2} with $\hat \theta = 1/2$ for some $\zeta > 0$:
\begin{equation}
	\norm{x^{t +2} - x^*} = O\left(\norm{x^{t}-x^*}^{1+\rho}\right),
	\norm{\nabla F_\phi\left(\modifyy{y}^{t +2}\right) } = O\left(\norm{\nabla
		F_\phi\left(\modifyy{y}^{t}\right)}^{1+\rho}\right).
	\label{eq:quad1}
\end{equation}
\item For $\rho = 0.69$ and $F_\phi$ satisfying
	\cref{eq:sharpness2} for some $\zeta > 0$ and $\hat \theta \geq 3/8$,
\begin{equation*}
\norm{x^{t +2} - x^*} = o\left(\norm{x^{t}-x^*}\right).
\end{equation*}
\end{enumerate}
\label{thm:super}
\end{theorem}
\begin{proof}
In our discussion below, $V_i$ and $U_i$ for $i \in \N$ are
respectively neighborhoods of $x^*$ and $y^*$.

Since $\phi$ is $\mathcal{C}^2$, there is $U_1$ of
$y^*$ such that
\begin{equation}
\phi\left( y \right) - \phi\left( y^* \right) =
\inprod{\nabla \phi(y^*)}{y-y^*} + O\left( \norm{y - y^*}^2 \right),
\forall y \in U_1.
\label{eq:taylor}
\end{equation}
Because the derivative of $\phi$ at $y^*$ is injective,
\cref{eq:taylor} implies
\begin{equation}
\norm{\phi\left( y \right) - \phi\left( y^* \right)} = \Theta\left(
	\norm{y-y^*} \right), \forall y \in U_1.
\label{eq:equiv}
\end{equation}

If the $t$th iteration is a TSSN step, we define $y^t$ to be the point
such that $\phi(y^t) = x^t$.
If either case in \cref{lemma:yt} holds and $q^t$ satisfies
\cref{eq:stop}, from that $\norm{y^t + q^t - y^*} = o(\norm{y^t -
y^*})$ we can find $U_2 \subset U$ such that $y^t\in U_2$ implies $y^t
+ q^t \in U_2$.
Take $U_3 \coloneqq U_1\cap U_2 \subset U$, for $y^t \in U_3$ and
$x^{t+1} = \phi(y^t + q^t)$, we get $y^t+q^t \in U_3 \subset U_1$, and
hence the following from \cref{eq:equiv}.
\begin{equation}
\norm{x^{t+1} - x^*} = \Theta\left(\norm{y^t+q^t - y^*}\right), \quad
\norm{x^t-x^*} = \Theta\left( \norm{y^t-y^*} \right).
\label{eq:propogate}
\end{equation}

On the other hand, consider the case in which the $t$th iteration is
a PG step.
As $\phi$ is $\mathcal{C}^2$ and $\phi(y^*) = x^*$, we can find $V_1$
such that $\phi(U_3) \supset V_1 \cap \M$.
From \cite[Lemma~3.2]{Wri12a}, there is $V_2$ such that $x^t\in V_2$
implies
\begin{equation}
\norm{x^{t+1} - x^*} = O(\norm{x^t-x^*}),
\label{eq:pgsize}
\end{equation}
so there is $V_3\subset V_2$ such that
$x^{t+1} \in V_1$ if $x^{t} \in V_3$.
Therefore, from \modify{\cref{thm:identify2} (applicable because we
have assumed \cref{eq:nod})}, there is $V_4$ such
that $x^t \in V_4$ implies $x^{t+1} \in \M$.
Take $V_5 \coloneqq V_4 \cap V_3$, then $x^t\in V_5$ implies $x^{t+1} \in V_1
\cap \M$, thus we can find $y^{t+1}\in U_3$ with $\phi(y^{t+1}) =
x^{t+1}$.

Now consider the first case in the statement.
If at the $t$th iteration we have $x^t\in V_5$ and have taken
\cref{eq:proxgrad1}, then $x^{t+1} \in V_1\cap \M$ with $x^{t+1} =
\phi(y^{t+1})$ for some $y^{t+1} \in
U_3$, so we can take a TSSN step at the $(t+1)$th iteration and
\begin{align}
\nonumber
\norm{x^{t+2} - x^*} &\stackrel{\cref{eq:propogate}}{=} \Theta \left(\norm{y^{t+1} + q^{t+1} -
y^*}\right) \stackrel{\cref{lemma:yt}}{=} O\left( \norm{y^{t+1} - y^*}^{1+\rho} \right)\\
&\stackrel{\cref{eq:propogate}}{=} O\left(
\norm{x^{t+1} - x^*}^{1+\rho} \right) \stackrel{\cref{eq:pgsize}}{=}
O\left( \norm{x^t - x^*}^{1+\rho} \right),
\label{eq:xconv1}
\end{align}
so there is $V_6 \subset V_5$ such that $x^t \in V_6$ implies $x^{t+2}
\in V_6$ as well and the superlinear convergence in \cref{eq:xconv1}
propagates to $t+2,t+4,\dotsc$.
We therefore see that $\norm{x^{t+2i+1} - x^*} =
O(\norm{x^{t+2i-1} - x^*}^{1+\rho})$ for $i \in \N$ as well,
so there is $V_7 \subset V_1$ such that $x^{t+2i-1} \in V_7 \cap \M$
implies $x^{t+2i+1} \in V_7 \cap \M$ as well.

Let $V = V_7 \cap V_6$, we see that $x^t \in V \cap \M$ implies
$x^{t+2} \in V \cap \M$ no matter we take PG or TSSN first, proving
the first equation in \cref{eq:quad1}.
The convergence of $\nabla F_\phi$ then follows from \cref{eq:gt} and
\cref{eq:gradbound}.
The superlinear convergence in the second case follows the same argument.
\qed
\end{proof}
Note that when $\rho=1$ in the first case, we obtain quadratic convergence.

The analysis in \cite{MCY19a} assumed directly \cref{eq:gradbound}
instead of \cref{eq:sharpness2}, together with a Lipschitzian Hessian
for $f$, under the setting of regularized optimization to get a
superlinear rate.
In the context of smooth optimization, our analysis is more general in
giving a wider range for superlinear convergence.
In particular, for \cref{eq:gradbound}, \cite{MCY19a} only allowed
$\tilde {\theta} = 1$, where $\tilde {\theta} \modify{\coloneqq} \hat \theta / (1 -
\hat \theta)$, whereas our result extends the range of superlinear
convergence to $\tilde {\theta} \geq 0.6$.

\begin{remark}
PCG returns the exact solution of \cref{eq:newton} in $d$ steps, where
$d$ is the dimension of $\M$, and each step involves only a bounded
number of basic linear algebra operations, so the running time of
\cref{alg:newton} is upper bounded.
Therefore, superlinear convergence of \cref{alg:vm2} in terms of
iterations, from \cref{thm:super}, implies that in terms of running time
as well.
This contrasts with existing PN approaches, as they all
require applying an iterative subproblem solver to \cref{eq:quadratic}
with increasing precision, which also takes increasing time per
iteration because \cref{eq:quadratic} has no closed-form solution.
\end{remark}

\ifdefined\arxiv
\else
\modify{
\subsection{Availability of Parameterization for the Manifold}
\label{subsec:chart}
The proposed \cref{alg:vm2} relies on the existence of the
parameterization $\phi$.
There are many widely-used regularizers in machine learning and signal
processing that are known to be partly smooth, and the corresponding
manifold can also be described by some simple parameterizations.
A summary of some of the most popular regularizers, including the
$\ell_1$-norm, $\ell_\infty$-norm, $\ell_{2,1}$-norm, $\ell_0$
pseudo-norm, nuclear norm (for matrix variables), and total variation,
and their associated manifolds can be found in
\cite[Chapter~5.2]{Lia16a}.

Except for those regularizers listed above, following the discussion
of \cite{MilM05a,Har06a}, several well-developed approaches can be
leveraged for the search of the parameterization in other partly
smooth regularizers.
The first one is the Riemannian optimization approach
(see, for example, \citep{AbsMS09a}), of which a projection from $\H$
to the tangent space of a manifold and a retraction map that projects
from the tangent space back to the manifold are the
workhorses\modifyy{.
Such} mappings for various manifolds have been constantly being developed.
The composite of the projection to the tangent space and the
retraction to the manifold can then be used as a reasonable
parameterization for our purpose.
The second approach is to rely on the $\mathcal{VU}$-decomposition
\citep{LemOS00a,MifS00a,MifS03a,MifS05a}.
It has been shown by \cite{MilM05a,Har06a} that the so-called fast
track in the $\mathcal{VU}$-decomposition for a convex and partly
smooth regularizer is equivalent to its corresponding active manifold,
and thus the developed descriptions for the fast track of various
functions can be utilized as another possibility for parameterization.
This corresponds to the tangential parameterization discussed in
\cite{MilM05a}.
These are also exactly the two parameterizations suggested by
\cite{LiaFP17a} in applying high-order acceleration steps for a family
of forward-backward splitting algorithms.
Another possible parameterization suggested by \cite{MilM05a} is the
projection parameterization that uses the Euclidean projection to map
a point in the tangent space back to the manifold.
}
\fi
\section{Numerical Results}
\label{sec:exp}
We conduct numerical experiments on $\ell_1$-regularized logistic
regression to support our theory, which is of the form in \cref{eq:f}
with
$\H = \R^d$ for some $d \in \N$,
\begin{equation}
	\Psi(x) = \lambda \norm{x}_1,\; f(x) = \sum_{i=1}^n \log\left(
	1 + \exp\left( -b_i \inprod{a_i}{x} \right) \right),
	\label{eq:l1lr}
\end{equation}
where $\lambda > 0$ decides the weight of the regularization and
$(a_i,b_i) \in \R^d \times \{-1,1\}$ for $i=1,\dotsc,n$ are the data
points.
Note that $\lambda \norm{x}_1$ is partly smooth at every $x \in \R^d$
relative to $\M_{x} \coloneqq \left\{ y \mid y_i = 0, \forall i \in
	\modifyy{J}_x \right\}$, where $\modifyy{J_x} \coloneqq \{i \mid x_i = 0\}$.
\modifyy{
Let $I$ be the identity matrix, and $J_x^C\coloneqq \{1,\dotsc,d\}
\setminus J_x$, then viewing from the definition of
$\M_x$ here, the parameterization we use at each iteration is
simply projecting $y_t \in \R^{|I_{x_t}|}$ back to $\R^d$.
Namely,}
\[
	\modifyy{\phi_t(y_t) \coloneqq I_{:,J_{x_t}^C}y_t.}
\]

We use public available real-world data sets listed in
\cref{tbl:data}.\footnote{Downloaded from
	\url{http://www.csie.ntu.edu.tw/~cjlin/libsvmtools/datasets/}.}
All experiments are conducted with $\lambda = 1$ in \cref{eq:l1lr}.
All methods are implemented in C++, and we set $\gamma = 10^{-4},
\beta = 0.5, T = 5$ throughout for \ipvm and \ipn.

\begin{table}[b]
	\caption{Data statistics.}
	\label{tbl:data}
\centering
\begin{tabular}{@{}lrr@{}}
	\toprule
Dataset & \#instances ($n$) & \#features ($d$)\\
	\midrule
\a9a & 32,561 & 123\\
\realsim & 72,309 & 20,958\\
\news & 19,996 & 1,355,191\\
\ijcnn1 & 49,990 & 22\\
\cov & 581,012 & 54\\
\rcvt & 677,399 & 47,226\\
\eps & 400,000 & 2,000\\
\webspam & 350,000 & 16,609,143\\
	\bottomrule
\end{tabular}
\end{table}
\subsection{Manifold Identification of Different Subproblem Solvers}
\label{subsec:subprobexp}
We start with examining the ability for manifold identification of
different subproblem solvers.
We run both \ipvm and the first stage of \ipn (by setting $S =
\infty$) and consider two settings for $H_t$.
The first is the L-BFGS approximation with a safeguard in
\cite{LeeLW19a}, and we set $m=10$ and $\delta = 10^{-10}$ in their
notation following their experimental setting.
The second is a PN approach in \cite{MCY19a} that uses
$H_t = \nabla^2 f\left( x^t \right) + c \norm{x^t - \prox_{\Psi}\left(
x^t - \nabla f\left( x^t \right) \right)}^\rho I$,
and we set $\rho = 0.5, c= 10^{-6}$ following their suggestion.
In both cases, we enlarge $H_t$ in \cref{alg:vm2} through $H_t
\leftarrow 2 H_t$.

We compare the following subproblem solvers.
\begin{itemize}
	\item SpaRSA \citep{WriNF09a}: a backtracking PG-type
		method with the initial step sizes estimated by the
		Barzilai-Borwein method.
	\item APG \citep{Nes13a}:
		See \cref{eq:ag}-\cref{eq:ag2}.
		We use a simple heuristic of restarting whenever the objective
		is not decreasing.
	\item Random-permutation cyclic coordinate descent (RPCD):
		Cyclic proximal coordinate descent with
		the order of coordinates reshuffled every epoch.
\end{itemize}

The results presented in \cref{tbl:result} show that all subproblem
solvers for \ipn can identify the active manifold, verifying
\cref{thm:identify2}.
Because the step sizes are mostly one in this experiment,
even solvers for \ipvm can identify the active manifold.
Among the solvers, RPCD is the most efficient and stable in
identifying the active manifold, so we stick to it in subsequent
experiments.

\begin{table}[b]
\centering
\caption{Outer iterations and time (seconds) for
different subproblem solvers to identify the active manifold.
\modify{For each \ipn variant, the fastest running time among all
subproblem solvers is boldfaced.}
}
\begin{tabular}{@{}llrrrrrr@{}}
	\toprule
Dataset& & \multicolumn{2}{c}{\a9a} & \multicolumn{2}{c}{\realsim}&
\multicolumn{2}{c}{\news} \\
\cmidrule{3-4}
\cmidrule{5-6}
\cmidrule{7-8}
& & Time & Iter.& Time & Iter.& Time & Iter.\\
\midrule
\multirow{3}{*}{\ipn-LBFGS} &SpaRSA & 0.7 & 229 & \textbf{5.0} & 185 & 10.3 & 136 \\
&APG & 1.3 & 408 & 11.8 & 408 & 16.7 & 230 \\
&RPCD & \textbf{0.6} & 190 & 5.7 & 165 & \textbf{9.6} & 131 \\
\midrule
\multirow{3}{*}{\ipvm-LBFGS} &SpaRSA & 0.9 & 238 & \textbf{5.9} & 225
& \textbf{6.1} & 129 \\
&APG & 1.8 & 617 & 12.5 & 418 & 11.9 & 237 \\
&RPCD & \textbf{0.8} & 238 & 6.1 & 174 & 7.8 & 149 \\
\midrule
\multirow{3}{*}{\ipn-Newton} &SpaRSA & \textbf{3.3} & 16 & 13.9 & 11 & 6.6 & 19 \\
&APG & 6.3 & 45 & 99.6 & 76 & 8.7 & 19 \\
&RPCD & 6.6 & 371 & \textbf{1.6} & 11 & \textbf{0.5} & 7 \\
\midrule
\multirow{3}{*}{\ipvm-Newton} &SpaRSA & \textbf{3.4} & 16 & 14.2 & 11 & 6.4 & 19 \\
&APG & 6.3 & 45 & 88.5 & 76 & 8.8 & 19 \\
&RPCD & 6.8 & 371 & \textbf{1.6} & 11 & \textbf{0.6} & 7 \\
\bottomrule
\end{tabular}
\label{tbl:result}
\end{table}
\subsection{Comparing \ipn with Existing Algorithms}

We proceed to compare \ipn with the following state of the art for
\cref{eq:f} using the relative objective value: $(F(x) - F^*) / F^*$.
\begin{itemize}
	\item LHAC \citep{SchT16a}: an inexact proximal L-BFGS method with
		RPCD for \cref{eq:quadratic} and a trust-region-like
		approach.
		Identical to our L-BFGS variant,
		we set $m=10$ for constructing $H_t$ in this experiment.
	\item NewGLMNET \citep{YuaHL12a}: a line-search PN with an RPCD
		subproblem solver.
\end{itemize}
IRPN \citep{MCY19a} is another PN method that performs slightly faster
than NewGLMNET, but their algorithmic frameworks are similar and the
experiment in \cite{MCY19a} showed that the running time of NewGLMNET
is competitive.
We thus use NewGLMNET as the representative because its code is
open-sourced.

For \ipn, we set $S=10$ and use both PN and L-BFGS variants with
RPCD in the first stage and \cref{alg:newton} with $\rho = 0.5, c =
10^{-6}$ in the second.
For PCG, we use the diagonals of $H_t$ as the preconditioner.
We use a heuristic to let PCG start with an iteration bound $T_0 = 5$,
double it whenever $\alpha_t = 1$ until reaching the dimension of
$\M$, and reset it to $5$ when $\alpha_t < 1$.
\modify{For the value of $S$, although tuning it properly might lead to even
better performance, we observe that the current setting already
suffices to demonstrate the improved performance of the proposed
algorithm.}

\modify{Results in \cref{fig:compare} show} the superlinear convergence
in running time of \ipn, while LHAC and NewGLMNET only exhibit linear
convergence.
We observe that for data with $n \gg d$, including \a9a, \ijcnn1,
\cov, and \eps, L-BFGS approaches are faster because $H_t p$ can be
evaluated cheaply (LHAC failed on \cov due to implementation issues),
and PN approaches are faster otherwise, so no algorithm is always
superior.
Nonetheless, for the same type of $H_t$, \ipn-LBFGS and \ipn-Newton
respectively improve state-of-the-art algorithms LHAC and NewGLMNET
greatly because of the fast local convergence, especially when the
base method converges slowly.

\begin{figure}[tb]
\centering
\begin{tabular}{ccc}
\begin{subfigure}[b]{0.29\textwidth}
	\includegraphics[width=\linewidth]{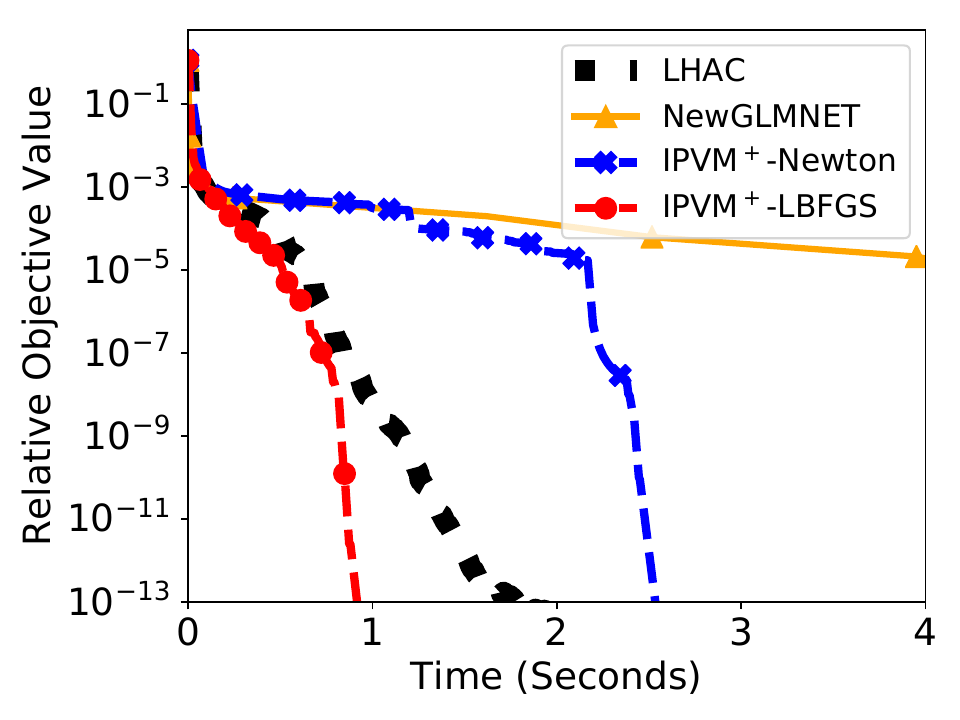}
	\caption{\a9a}
\end{subfigure}&
\begin{subfigure}[b]{0.29\textwidth}
	\includegraphics[width=\linewidth]{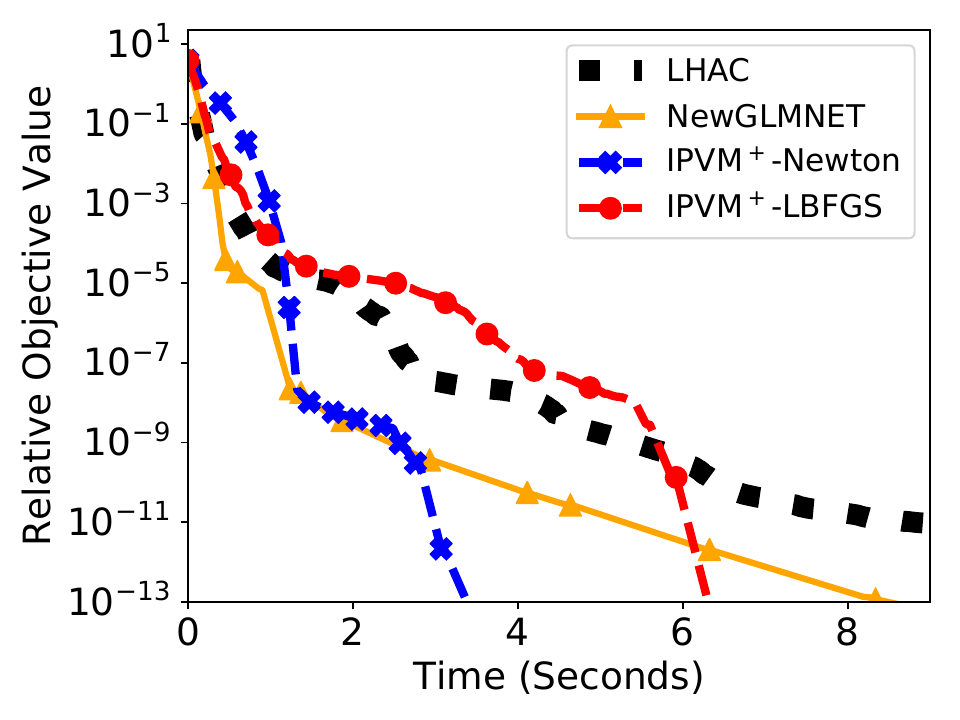}
	\caption{\realsim}
\end{subfigure}&
\begin{subfigure}[b]{0.29\textwidth}
	\includegraphics[width=\linewidth]{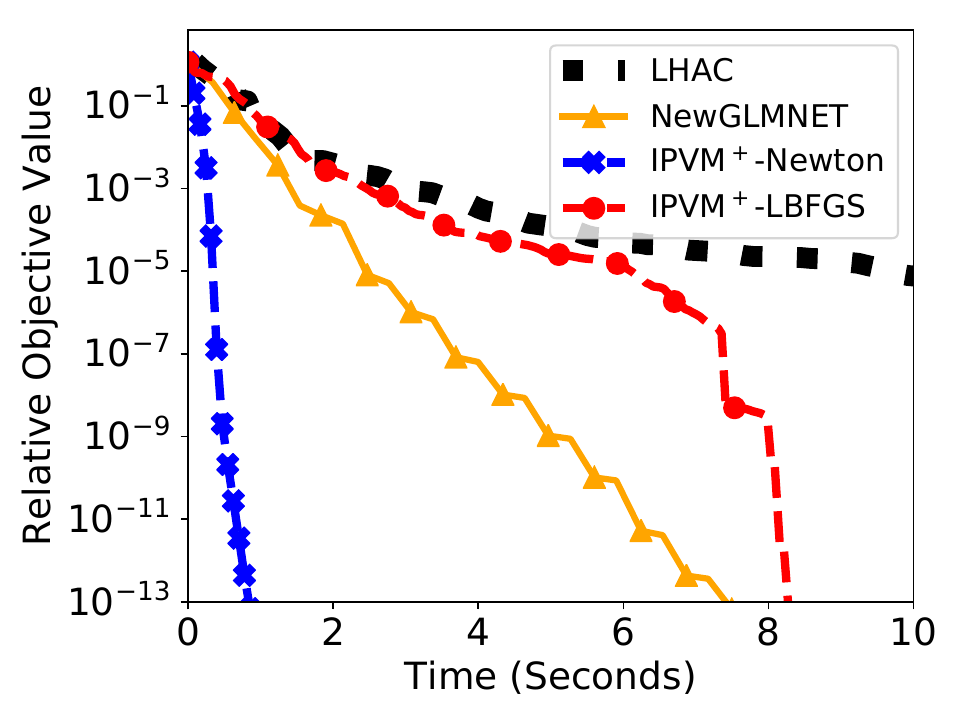}
	\caption{\news}
\end{subfigure}\\
\begin{subfigure}[b]{0.29\textwidth}
	\includegraphics[width=\linewidth]{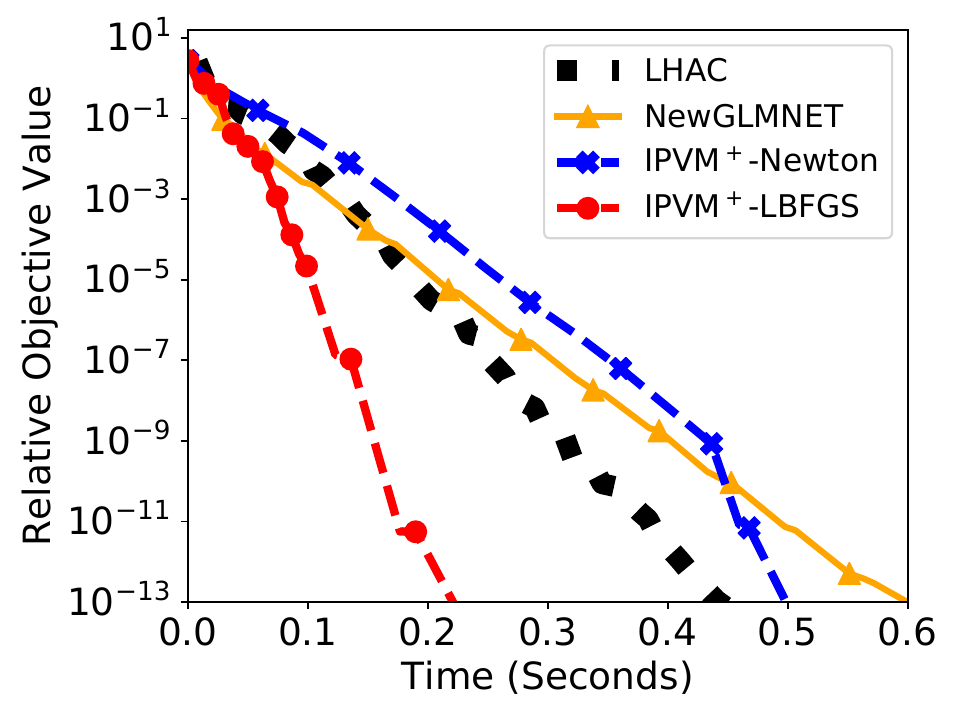}
	\caption{\ijcnn1}
\end{subfigure}&
\begin{subfigure}[b]{0.29\textwidth}
	\includegraphics[width=\linewidth]{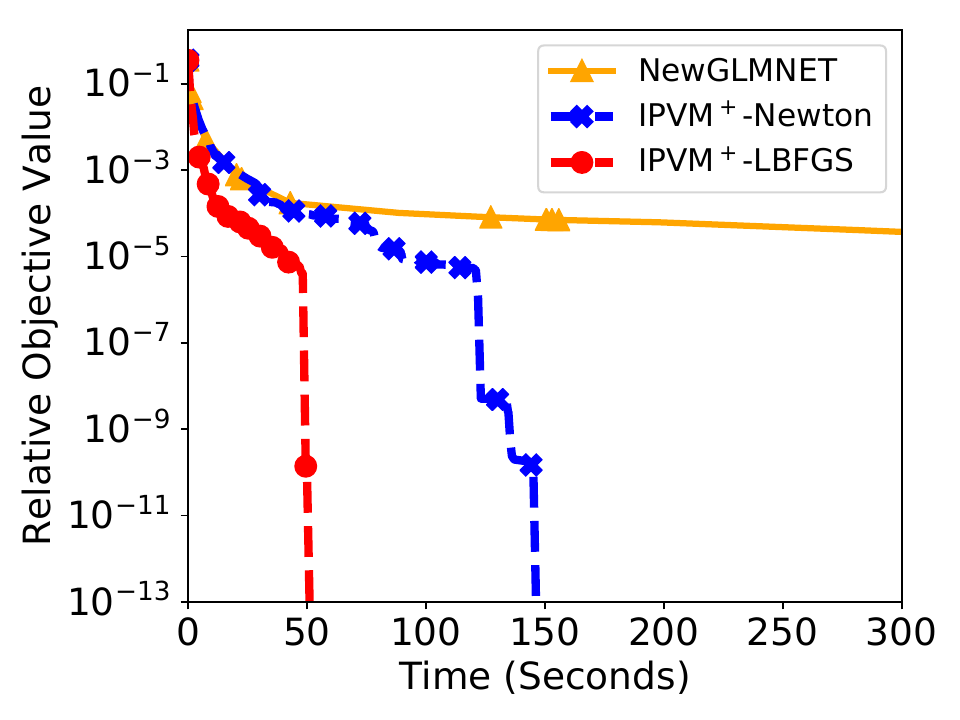}
	\caption{\cov}
\end{subfigure}&
\begin{subfigure}[b]{0.29\textwidth}
	\includegraphics[width=\linewidth]{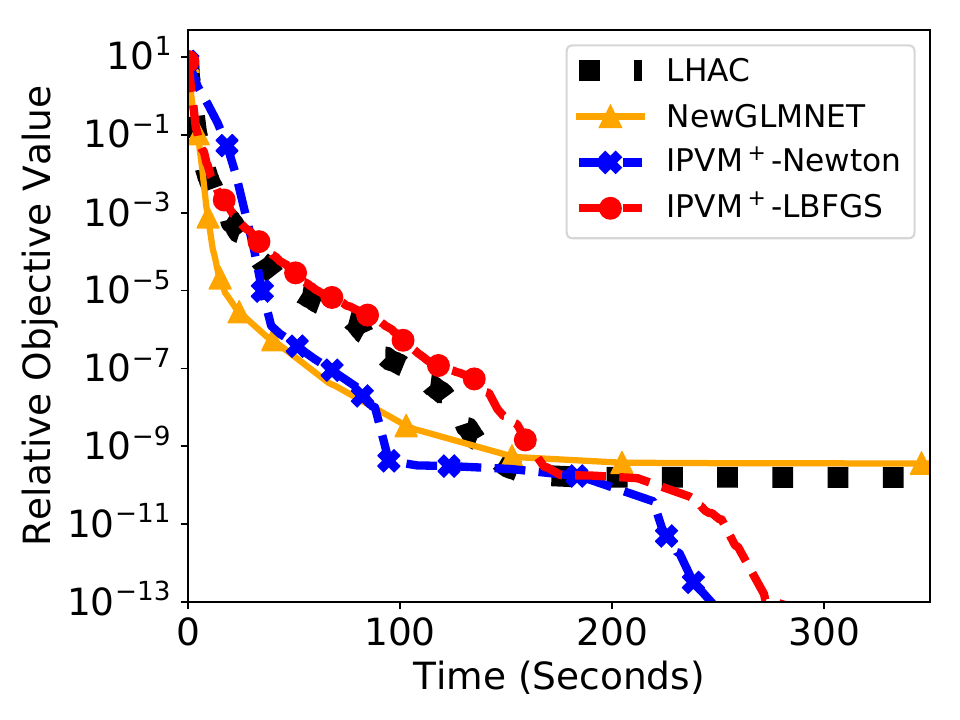}
	\caption{\rcvt}
\end{subfigure}\\
\begin{subfigure}[b]{0.29\textwidth}
	\includegraphics[width=\linewidth]{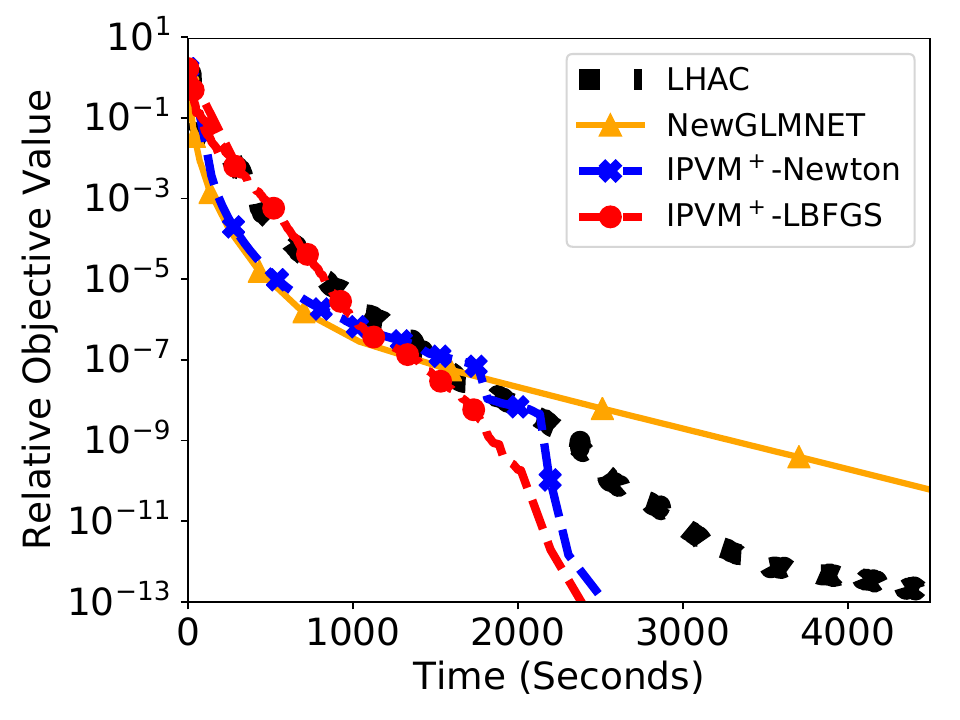}
	\caption{\eps}
\end{subfigure}&
\begin{subfigure}[b]{0.29\textwidth}
	\includegraphics[width=\linewidth]{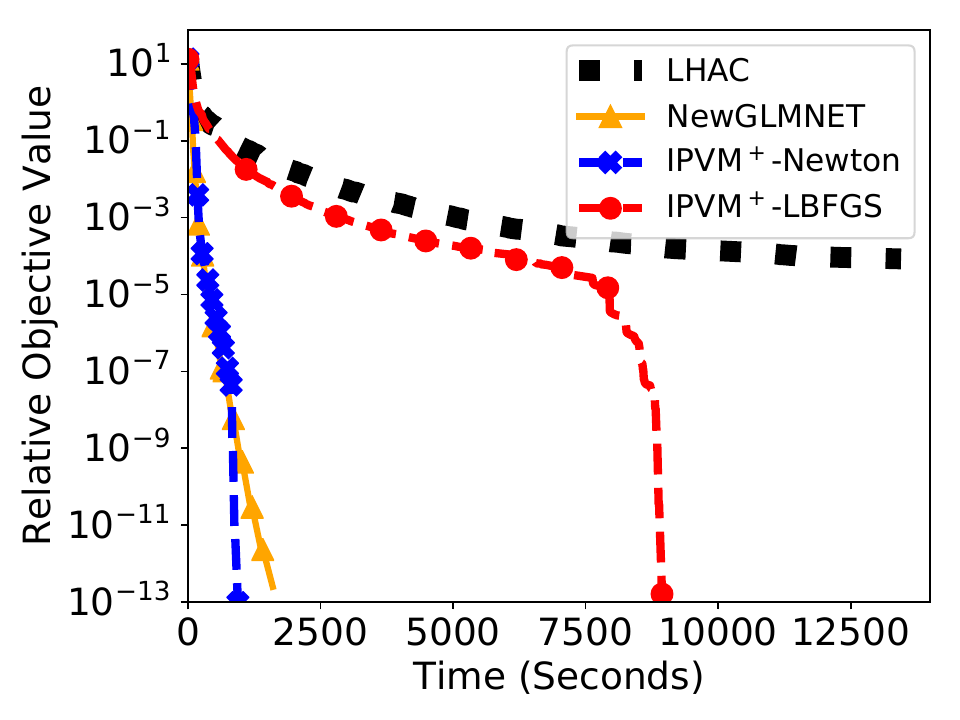}
	\caption{\webspam}
\end{subfigure}&
\begin{subfigure}[b]{0.29\textwidth}
	\includegraphics[width=\linewidth]{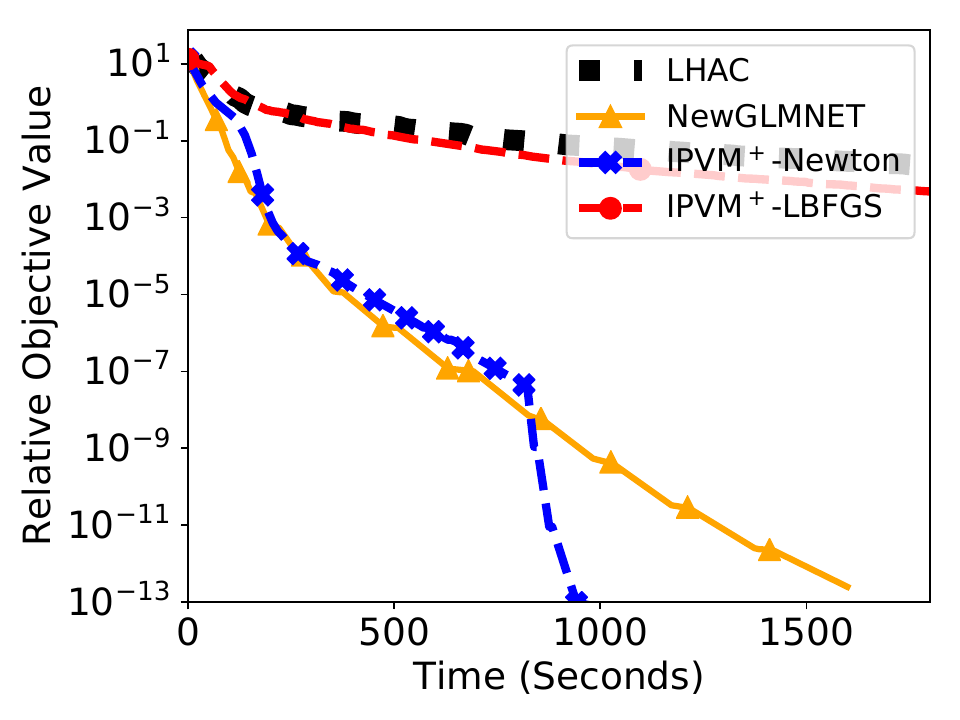}
	\caption{\webspam (finer scale)}
\end{subfigure}
\end{tabular}
	\caption{Comparison of Different Algorithms. For \cov, LHAC is not shown
	because it failed to converge.}
	\label{fig:compare}
\end{figure}

\section{Conclusions}
\label{sec:con}
In this paper, we showed that for regularized problems with a
partly smooth regularizer, inexact successive quadratic approximation
is essentially able to identify the active manifold because a mild
sufficient condition is satisfied by most of commonly-used subproblem
solvers.
An efficient algorithm \ipn utilizing this property is proposed to
attain superlinear convergence on a wide class of degenerate problems
in running time, greatly improving upon state of the art for
regularized problems that only exhibit superlinear convergence in
outer iterations.
\ifdefined\arxiv
Numerical evidence illustrated that \ipn significantly improves the running time of
state of the art for regularized optimization.
\else
Numerical evidence illustrated that \modify{\ipn improves} the running time of
state of the art for regularized optimization.
\fi

\section*{Acknowledgements}
I thank Steve Wright for pointing to me the direction of manifold
identification.  It is my great pleasure to have worked with him during my
PhD study, and this work is extended from my research in that period.  I
also thank Michael Ferris for the inspiration of the sharpness
condition.
\bibliographystyle{spbasic}
\ifdefined\arxiv
\bibliography{header-arxiv,vmmanifold}
\else
\bibliography{header-submit,vmmanifold}
\fi

\appendix
\section{Proof of \cref{thm:rate}}
\label{app:proof1}
\begin{proof}

For any $t \geq 0$, \cite[Lemma~5]{LeeW18a} gives
\begin{align}
Q^*_t \leq  \min_{\lambda \in [0,1]}\,\lambda(F^* -
F\left(x^t\right)) + \frac{\lambda^2 }{2} \|H_t\| \dist(x^t,\Omega)^2.
\label{eq:Q0}
\end{align}
Combining \cref{eq:Q0} and \cref{eq:sharpness} then leads to
\begin{align}
Q^*_t \leq -\lambda \delta_t + \frac{\lambda^2
	\|H_t\|\delta_t^{2\theta}}{2\zeta^2},\;\;
\forall\lambda \in [0,1].
\label{eq:Q}
\end{align}
Through \cref{eq:armijo}, \cref{eq:subprob}, and
\cref{eq:multiplicative},
we can further deduce from \cref{eq:Q} that
\begin{align}
&~\delta_{t+1} - \delta_t \leq \alpha_t \gamma \left( 1 - \eta
	\right) Q_t^*
	\leq \alpha_t \gamma (1 - \eta) \left( - \lambda \delta_t +
	\frac{\lambda^2 M \delta_t^{2\theta}}{2 \zeta^2} \right),
	\quad \forall \lambda \in [0,1].
\label{eq:improve}
\end{align}

When $\theta = 1/2$, clearly \cref{eq:Q}-\cref{eq:improve} imply
\begin{equation}
	\delta_{t+1} \leq \min_{\lambda \in [0,1]}\,\delta_t \left( 1 - \alpha_t \gamma \left( 1 - \eta
	\right) \left(\lambda - \frac{\lambda^2 \|H_t\|}{2 \zeta^2} \right)
	\right).
	\label{eq:linearrate}
\end{equation}
Setting $\lambda = \min\left\{ 1, \zeta^2\|H_t\|^{-1} \right\}$
in \cref{eq:linearrate} then proves \cref{eq:qlinear}.

For $\theta \neq 1/2$, we will use \cref{eq:improve} with
\begin{equation}
	\lambda = \min\left\{ 1, \zeta^2\delta_t ^{1 -
	2\theta}M^{-1} \right\}.
	\label{eq:lambda}
\end{equation}
The case in which $\lambda = 1$ in \cref{eq:lambda} suggests that
	$M\delta_t^{2\theta - 1}\zeta^{-2} \leq 1$
and corresponds to the following linear decrease.
\begin{align*}
\delta_{t+1} - \delta_t
\leq
\alpha_t \gamma (1 - \eta) \delta_t \left( -1 + \frac{M
	\delta_t^{2 \theta - 1}}{2 \zeta^2}\right)
\leq -\delta_t \frac{\alpha_t \gamma (1 - \eta) }{2},
\end{align*}
which is exactly \cref{eq:weaksharp}.
When $\theta > 1/2$, $2 \theta - 1 > 0$ so $M\delta_t^{2\theta -
1}\zeta^{-2} \leq 1$ if and only if \cref{eq:bound} fails,
and when $\theta < 1/2$, it is the opposite case because $2 \theta - 1
< 0$.

When $\lambda = \zeta^{2}\delta_t ^{1 -2
\theta} M^{-1}$ in \cref{eq:lambda}, \cref{eq:improve} leads to
\begin{equation}
\delta_{t+1}  \leq \delta_t -
\frac{\alpha_t \left( 1 - \eta \right)\gamma\zeta^2}{2M}\delta_t^{2 - 2 \theta}.
\label{eq:fixed2}
\end{equation}
For $\theta > 1/2$, this happens when \cref{eq:bound} holds.
The rate \cref{eq:earlylinear} is then obtained from
\cref{eq:fixed2} by $-\delta_t^{1 - 2\theta} \leq - \xi^{1 - 2\theta}$
for all $t$, implied by $x^0 \in \levbet$ and the monotonicity of
$\{\delta_t\}$.
For $\theta < 1/2$, since $1 - 2\theta > 0$, $\lambda < 1$
if $\delta_t$ is small enough to fail \cref{eq:bound}.
We have $2 - 2\theta > 1$,
so \cref{eq:fixed2} and Lemma~6 in \cite[Chapter~2.2]{Pol87a}
imply \cref{eq:sublinear}.
\qed
\end{proof}

\section{Proof of \cref{thm:iterate}}
\label{app:proof}
From \cref{thm:rate}, we have gotten a faster convergence speed of the
objective.
We then need to relate it to the decrease of $\hat Q_t$ and $Q_t^*$
before proving \cref{thm:iterate}.

\begin{lemma}
Consider \cref{alg:vm}.
If $\Omega\neq \emptyset$, $f$ is $L$-smooth for $L > 0$, $\Psi
\in \Gamma_0$,
and there are $m > 0$ and $\eta \in [0,1)$ such that $H_t \succeq m$
	and $p^t$ satisfies \cref{eq:subprob} with
	\cref{eq:multiplicative} for all $t$,
then
\begin{gather}
\hat Q_t - Q_t^* \leq
\frac{\eta\delta_t}{\gamma \bar \alpha (1 -
\eta)},\quad
-Q_t^* \leq \frac{\delta_t}{\gamma \bar \alpha (1 - \eta)},\quad
\forall t \geq 0,
\label{eq:pstar}
\end{gather}
where $\bar \alpha >0$ is the constant in \cref{thm:rate}.
\label{lemma:p}
\end{lemma}
\begin{proof}
From \cref{eq:subprob}, \cref{eq:multiplicative}, and $Q_t(0)
= 0$, we know that
\begin{align}
\label{eq:Q1}
- \hat Q_t \geq - (1 - \eta) Q_t^*.
\end{align}
By \cref{eq:Q1}, \cref{eq:armijo}, $\delta_t \geq 0$, and
$\alpha_t \geq \bar \alpha$ from \cref{thm:rate}, the second
inequality of \cref{eq:pstar} is proven by
\begin{equation}
	\label{eq:ptar0}
\delta_t \geq \delta_t - \delta_{t+1} \geq - \gamma \alpha_t \left( 1 - \eta \right)
Q_t^*
\geq - \gamma \bar \alpha \left( 1 - \eta \right) Q_t^*.
\end{equation}
	
For the first inequality,
the case of $\eta = 0$ holds trivially.
If $\eta \in (0,1)$,
\cref{eq:Q1} implies
\begin{align}
\label{eq:Q2}
-\eta \hat Q_t \geq
(1 - \eta) \left( \hat Q_t - Q_t^* \right)
\quad \Rightarrow \quad
-\hat Q_t \geq
\frac{1 - \eta}{\eta} \left( \hat Q_t - Q_t^* \right).
\end{align}
Following the same logic for obtaining \cref{eq:ptar0},
we get from \cref{eq:armijo} and \cref{eq:Q2} that
	\[\eta \delta_t \geq \gamma \alpha_t (1 - \eta) \left( \hat Q_t
- Q_t^* \right) \geq
	\gamma \bar \alpha (1 - \eta) \left( \hat Q_t - Q_t^*
\right).
	\tag*{\qed}
\]
\end{proof}

\begin{proof}[Proof of \cref{thm:iterate}]
	\cref{lemma:p} together with \cref{eq:qg} implies that
\begin{equation}
	\norm{p^t - p^{t*}} = O\left(\delta_t^{\frac12}\right),\quad
	\norm{p^{t*}} =\norm{0 - p^{t*}} = O\left(\delta_t^{\frac12}\right).
	\label{eq:pnorm}
\end{equation}
Through the triangle inequality, we have from
\cref{eq:pnorm} and $\alpha_t \leq 1$ that
\begin{equation}
	\norm{x^{t+1} - x^t} = \norm{\alpha_t p^t} = \alpha_t \norm{p^t}
	\leq \norm{p^t} \leq \norm{p^{t*}} + \norm{p^t - p^{t*}}
	= O\left( \delta_t^{\frac{1}{2}} \right).
	\label{eq:normdiff}
\end{equation}
Condition \cref{eq:sharpness} implies that for any $x^0$, $\dist(x,
\Omega)$ is upper-bounded for all $x \in \levv{F(x^0)}$, as argued in
\cite{Bol17a}.
Thus, \cite[Theorem~1]{LeeW18a} implies that $\levbet$ is reached in
finite iterations.
As the asymptotic convergence is concerned, we assume
$x^0 \in \levbet$ without loss of generality.

We now apply \cref{thm:rate} and separately consider $\theta
\in [1/2,1]$ and $\theta \in (1/4,1/2)$.
When $\theta \in [1/2,1]$, $\{\delta_t\}$ converges $Q$-linearly to
$0$,
so there exist $t_0 \geq 0$ and $\tau \in [0,1)$ such that
\[
	\delta_{t+1} \leq \tau \delta_t, \quad \forall t \geq t_0.
\]
Thus for any $t_1 \geq t_2 \geq t_0$, we have from the triangle
inequality and \cref{eq:normdiff} that
\begin{align*}
&~\norm{x^{t_1} - x^{t_2}} \leq \sum_{t=t_2}^{t_1-1} \norm{x^{t+1} -
x^t} \\
= &~O\left( \sum_{t=t_2}^{t_1-1} \sqrt{\tau^{t - t_0}\delta_{t_0}} \right) =
O\left( \sqrt{\delta_{t_0}}\sqrt{\tau^{t_2 -
t_0}} \left(1 - \sqrt{\tau^{t_1 -
t_2}}\right)(1 - \sqrt{\tau})^{-1}\right),
\end{align*}
whose right-hand side converges to $0$ as $t_2$ approaches infinity.
Therefore, $\{x^t\}$ is a Cauchy sequence and converges to a point
$x^*$ as $\H$ is a complete space.

Similarly, for $\theta \in (1/4,1/2)$,
\cref{thm:rate} implies that there is $t_0 \geq 0$ such that
\[
	\sqrt{\delta_t} = O\left( t^{-\frac{1}{2(1 - 2\theta)}} \right)
	= O\left(t^{-\tau}\right),\quad \forall t \geq t_0
\]
with $\tau = 1 / (2(1 - 2\theta)) \in (1,\infty)$.
Therefore,
\[
	\sum_{t=t_2}^{\infty} t^{-\tau} = O\left( t_2^{-\tau + 1}
	\right),\quad \forall t_2 \geq t_0,
\]
and the right-hand side converges to $0$.
We thus have from above that
$\sum_{t=t_2}^{t_1-1}\sqrt{\delta_t}$ converges to $0$ as $t_2$
approaches infinity,
and \cref{eq:normdiff} shows that so does $\sum_{t=t_2}^{t_1-1}
\norm{x^{t+1}- x^t}$.
This shows that $\{x^t\}$ is a Cauchy sequence and hence
converges to a point $x^* \in \H$.

It remains to show that $x^* \in \Omega$.
Since $f$ is continuous and $\Psi$ is lower semicontinuous, $F$ is
also lower semicontinuous.
Therefore, we have that
\begin{equation}
	\lim\inf_{t\rightarrow \infty} F\left( x^t \right) \geq F\left( x^* \right) \geq F^*.
	\label{eq:liminf}
\end{equation}
However, \cite[Theorem~1]{LeeW18a} implies $\lim_{t \rightarrow
\infty} F(x^t) = F^*$, which together with \cref{eq:liminf} shows that
$F(x^*) = F^*$, and hence $x^* \in \Omega$.
\qed
\end{proof}

\modify{
\section{Proof of \cref{thm:rate2}}
\label{app:proof2}
\begin{proof}
We first note that \cref{eq:proxgrad1} can be seen as solving
\cref{eq:quadratic} exactly with $H_t = \hat L I$, and it is widely-known
that for $\hat L \geq L$, \cref{eq:armijo} is always satisfied with
$\alpha_t = 1$.
Therefore, under \cref{assum:Hbound}, we have
\begin{equation}
	\tilde M \succeq H_{k_t} \succeq \tilde m, \quad \forall t \geq 0,
	\label{eq:Hbound}
\end{equation}
and \cref{eq:subprob} with \cref{eq:multiplicative}  holds for all
$k_t$.
For the MO steps, we note that it does not increase the objective
value, so for any computation involving the objective change, we can
safely skip them.
That $k_t \leq 2t$ for any $t$ is clear, since we do not conduct two
consecutive MO steps.

\begin{enumerate}
\item The part of $G_t = 0$ if and only if $0 \in \partial F(x^t)$ is
well-known. See, for example, \cite[Lemma~7]{LeeW18a}.
Thus, the claim of stationarity of the limit points follow directly from
the claim that $G_{k_t} \rightarrow 0$.

Now let us prove that $G_{k_t} \rightarrow 0$ and the upper bound for
$\min_{0 \leq k \leq t}  \norm{G_{k_t}}^2$.
For the $k_i$th iteration, we get from \cref{eq:armijo}, and that
\cref{eq:multiplicative} indicates that $\hat Q_{k_i} < 0$, that
\begin{equation}
	-\gamma \hat Q_{k_i} \leq
	F\left( x^{k_i} \right) - F\left( x^{k_i+1}\right).
	\label{eq:Qdecrease}
\end{equation}
By applying \cref{eq:multiplicative} and \cref{eq:subprob} to
\cref{eq:Qdecrease}, we further get
\[
	- \gamma (1 - \eta) Q_{k_i}^* \leq 
	F\left( x^{k_i} \right) - F\left( x^{k_i+1}\right).
\]
By summing \cref{eq:Qdecrease} over $i=0,\dotsc,t$, we get
\begin{equation}
\label{eq:Qsum2}
- \gamma (1 - \eta) \sum_{i=0}^t Q_{k_i}^* \leq 
\sum_{i=0}^t F\left( x^{k_i} \right) - F\left( x^{k_i+1}\right)
\leq
F\left( x^{k_0} \right) - F\left( x^{k_t+1} \right) \leq F(x^0) - F^*,
\end{equation}
where the penultimate inequality is from that $F(x^t)$ is
nonincreasing.
Following \cite[Lemma~7]{LeeW18a} and \cite[Lemma~3]{TseY09a}, we then
have from \cref{eq:Hbound} that
\begin{equation}
	\frac{F(x^0) - F^*}{\gamma (1 - \eta)} \geq \sum_{i=0}^t -Q_{k_i}^* \geq
	\sum_{i=0}^t \frac{2 \tilde m}{\tilde M^2 \left( 1 + \tilde m^{-1} +
	\sqrt{1 - 2 \tilde M^{-1} + \tilde m^{-2}} \right)^2}
	\norm{G_{k_i}}^2.
	\label{eq:penultimate}
\end{equation}
Since the above is true for any $t \geq 0$, it shows that
$\{\norm{G_{k_t}}^2\}$ is summable, so $\norm{G_{k_t}}^2$ converges to
$0$, which implies that $G_{k_t} \rightarrow 0$.
The rate is then obtained from \cref{eq:penultimate} by noting that
\[
	\sum_{i=0}^t \norm{G_{k_i}}^2 \geq \left( t+1 \right) \min_{0 \leq
	i \leq t} \norm{G_{k_i}}^2.
\]

\item This part follows directly from the proof of
	\cite[Theorem~1]{LeeW18a} by replacing $t$ by $t_k$, noting that
	MO steps do not increase the objective value so $\delta_{k_{t+1}}
	\leq \delta_{k_t+1}$ for all $t$, and using the bound
	\[
		\inprod{x^t - P_{\Omega}(x^t)}{H_t (x^t - P_{\Omega}(x^t))}
		\leq \tilde M \norm{x^t - P_{\Omega}(x^t}^2 \leq M R_0^2,\quad
		\forall t \geq 0
	\]
	from \cref{eq:R0}.

\item This part follows from the proof of \cref{thm:rate} again by
	using \cref{eq:Hbound} and that
	$\delta_{k_{t+1}} \leq \delta_{k_t+1}$ for all $t$.
\end{enumerate}

\end{proof}
}

\end{document}